\documentclass[letterpaper,10pt, journal,twoside]{IEEEtran}  % Comment this line out if you need a4paper
\pdfoutput=1

\IEEEoverridecommandlockouts                              % This command is only needed if 
\usepackage{amsmath,amssymb,color,cite}%,showkeys}
\usepackage{graphicx}
\usepackage{subfigure}

\usepackage{latexsym}
\usepackage{algorithm}
\usepackage{algpseudocode}

\makeatletter
\algnewcommand{\LineComment}[1]{\Statex \hfill \(\triangleright\) #1}
\makeatother

\usepackage{verbatim}
\usepackage{cite}
\usepackage{amsmath,amsthm}

\renewcommand{\theenumi}{(\roman{enumi})}

\usepackage{pgf}
\usepackage{pgfplots,tikz}
\usetikzlibrary{tikzmark}
\usepackage[font=footnotesize]{caption}
\usepackage{arydshln}

\usepackage{amsfonts}
\usepackage{amssymb}
\usepackage{amsbsy}
\usepackage{bm}
\usepackage{bbm}

\renewcommand{\algorithmicrequire}{\textbf{Input: }}
\renewcommand{\algorithmicensure}{\textbf{Output: }}

\DeclareMathOperator{\im}{im}

\DeclareMathOperator{\rank}{rank}

\newcommand{\norm}[1]{\ensuremath{\| #1 \|}}

\newcommand\oprocendsymbol{\hbox{$\bullet$}}
\newcommand\oprocend{\relax\ifmmode\else\unskip\hfill\fi\oprocendsymbol}

\let\leq\leqslant
\let\geq\geqslant
\let\emptyset\varnothing

\newcommand{\R}{\mathbb R}
\newcommand{\N}{\mathbb N}
\newcommand{\bT}{\mathbb T}
\newcommand{\calM}{\ensuremath{\mathcal{M}}}
\newcommand{\calN}{\ensuremath{\mathcal{N}}}

\newcommand{\calP}{\ensuremath{\mathcal{P}}}

\newcommand{\calX}{\ensuremath{\mathcal{X}}}

 \newcounter{todocounter}
 \usepackage[colorinlistoftodos]{todonotes}

\newtheorem{theorem}{Theorem}[section]

\newtheorem{lemma}[theorem]{Lemma}
\newtheorem{corollary}[theorem]{Corollary}

\theoremstyle{remark}
\newtheorem{remark}[theorem]{Remark}

\theoremstyle{definition}
\newtheorem{assumption}{Assumption}
\newtheorem{definition}[theorem]{Definition}

\newtheorem{problem}{Problem}

\renewcommand{\baselinestretch}{1.02}
\allowdisplaybreaks

\title{Data-driven mode detection and stabilization of unknown
  switched linear systems\thanks{A preliminary version of this work
    appeared as~\cite{JE-SL-SM-JC:22-cdc} at the IEEE Conference on
    Decision and Control. This work was partially supported by AFOSR
    Award FA9550-19-1-0235.}}

\author{Jaap Eising${}^*$ \quad Shenyu Liu${}^*$ \quad Sonia
  Mart\'{\i}nez \quad Jorge Cort\'{e}s\thanks{${}^*$Both authors
    contributed equally. Jaap Eising is with
    the Department of Information Technology and Electrical
    Engineering at ETH Z\"{u}rich, Switzerland,
    \texttt{jeising@ethz.ch}. Shenyu Liu is with
    the School of Automation, Beijing Institute of Technology, China,
    \texttt{shenyuliu@bit.edu.cn}.  Sonia Mart\'{\i}nez and Jorge
    Cort\'{e}s are with the Department of Mechanical and Aerospace
    Engineering, University of California, San Diego,
    \texttt{\{soniamd,cortes\}@ucsd.edu}.}}

\begin{document}

\maketitle
	
\begin{abstract}
  This paper considers the stabilization of unknown switched linear
  systems using data.  Instead of a full system model, we have access
  to a finite number of trajectories of each of the different modes
  prior to the online operation of the system.  On the basis of
  informative enough measurements, we design an online switched controller
  that alternates between a mode detection phase and a stabilization
  phase.  Since the currently-active mode is unknown, the
  controller employs online measurements to determine
  it by implementing computationally efficient tests that check
  compatibility with the set of systems consistent with the
  pre-collected measurements. The stabilization phase applies a
  stabilizing feedback gain corresponding to the identified active
  mode and monitors the evolution of the associated Lyapunov function
  to detect switches. When a switch is detected, the controller
  returns to the mode-detection phase.  Under average dwell- and
  activation-time assumptions on the switching signal, we show that
  the proposed controller guarantees a {\color{black}practical} stability
  property of the closed-loop switched system. Various simulations illustrate
  our results.
\end{abstract}
% Lack of need to identify modes is attractive (we still need to
% detect it)

% Data informativity for specific task -- e.g., informativity tailored
% to specific goal.

% Made more challenging by the fact that, in switched system, we do
% not know what specific mode is generating the data.

\section{Introduction}
Switched linear systems have long been of interest to the systems and
control community. Such systems consist of several modes and a logic
rule which governs the switching between them.  Many real-world
applications are naturally modeled as switched systems, where the
plant switches modes due to design specifications, recurrent
environmental effects, human behavior, or a combination of
thereof. From a dynamic perspective, switched linear systems exhibit
far more complex behavior than linear systems and this makes the
design of stabilizing controllers and their analysis challenging.
Most design approaches build on the model-based paradigm, where a
model of the system and each of its modes is available to solve the
stabilization problem.  In practice, this often requires considerable
modeling effort through system identification.  Motivated by these
observations, we adopt an {\color{black} hybrid online-offline}
approach based on data informativity to synthesize a controller that
jointly deals with the (partial) identification and (robust)
stabilization of the unknown switched system.

\emph{Literature review:} The control of switched systems is a
well-studied field, the complexity of which requires a wide range of
stabilization and analysis techniques, see
e.g.,~\cite{DL:03,HL-PJA:09,FB-SM-CS:07}. The literature typically
addresses two primary types of stabilization problems concerning
switched systems. In the first scenario, the switching signal itself
is regarded as the control input, either in an open- or closed-loop
setting, with the objective of stabilizing the
system~\cite{MSB:98,JG-PC:06,PC-JG-AA:08}. In the second scenario, the
switching signal is predetermined, and the focus shifts to designing
an external input to control the system. {\color{black} When the
  switching signal can be observed in real time by the controller,
  methods can be delineated between those working for switching
  signals with upper bounds on the switching
  frequency~\cite{YS:11switched} and those for arbitrary
  switching~\cite{JWL-GED:06,DZ-LA-JD-QZ:18}. On the other hand, when
  the switching signal is unknown to the controller, control methods
  which rely on common Lyapunov functions~\cite{MH-FF-GD:13} or
  path-complete Lyapunov functions~\cite{MDR-TAL-MJ-RMJ:24} are
  developed.} It is important to note that Lyapunov approach requires
specific structural attributes of the system dynamics and hence such
methods do not apply in every case. To address this limitation and
stabilize a broader class of switched systems with unknown switching
signals, adaptive estimation techniques \cite{DC-LG-YL-YW:05} or
switch detection mechanisms become necessary. Inspired by the
literature on fault detection~\cite{MTR-AQK-GM-MA:16,LL-SXD-YN-JQ:22}
and state estimation~\cite{XZ-HL-JZ-HL:15} for switched systems, we
can employ Lyapunov-based methods to detect unknown
switches. Nevertheless, all the aforementioned works usually require a
precise model of the modes of the switched system. In order to address
uncertainties in the model, several works have employed robust and
adaptive techniques for stabilization.  For
instance,~\cite{HL-PJA:07,LIA-US:13} develops a stabilizing controller
of the switched system regardless of the realization of certain
parameters within a given range.  The work~\cite{SY-BDS-SB:17}
proposes an alternative approach to the same problem based on adaptive
controllers. Another relevant angle is the use of switched controllers
in~\cite{BDOA-TSB-FDB-JH-DL-ASM:00} for robustly stabilizing
non-switched systems.

Robust and adaptive methods still rely on a nominal model and, as
such, require some type of system identification. To circumvent this,
and leverage the development of new computation and data acquisition
methods, there is a growing adoption of data-driven techniques. 
%A number of different types of switched systems, distinguished by the
%type of switching signal, have been the focus of such efforts.
The works~\cite{CZ-MG-JZ:19,AK:20arXiv,AS-CA-FG:23arXiv} address
  the first type of problem, wherein the switching signal itself
  serves as the controlling element. The second scenario, stabilizing
  an arbitrarily switched system via an external input, is
  investigated in~\cite{TD-MS:18,TD-MS:22} when the switching signal
  is known, and in~\cite{ZW-GOB-RMJ:21arXiv} when the switching signal
  is unknown. The recent work~\cite{VB-SF:20,XW-JS-GW-FA-JC:2022} also
  delves into the data-based design of switching controllers.

%The aforementioned works are based on extending known model-based
%methods to the data-driven setup. 
An alternative approach that has
recently gained traction is based on Willems' fundamental
lemma~\cite{JCW-PR-IM-BLMDM:05}.
Within this broader context, the work~\cite{MR-CDP-PT:22} proposes a
stabilizing controller for an unknown switched system, which is found
only on the basis of noiseless measurements of the currently active
mode of the system. Our treatment here leverages the informativity
approach to data-driven control, as introduced in
\cite{HJVW-JE-HLT-MKC:20} and recently extended in
\cite{HJVW-MKC-JE-HLT:22,JE-JC:23-tac}. In simple terms, this method
characterizes the set of systems that are compatible with collected
data, and produces a robust controller that stabilizes \textit{all}
the systems in this class. Within this framework, the
work~\cite{MB-SG-JC:22} considers the problem of data-driven
stabilization in the situation where the switching signal is
determined by the controller.

\emph{Statement of contributions:} We consider the problem of
stabilizing a switched system subject to an unknown switching signal
and whose modes are unmodeled.
Given that the switching signal is unknown, we require all modes to be
stabilizable.  Using the informativity framework, we first use
pre-collected measurements to determine separate robust stabilizing
feedback gains for each mode. Our online stabilizing controller design
employs online data to switch between a \textit{mode detection phase}
and a \textit{stabilization phase}. In the former, we use the most
recent online measurements to uniquely determine the mode of the
system currently active. We provide an algorithm that, for each time
instance, applies bounded inputs to excite the system and evaluates if
the new measurements obtained are compatible with the set of systems
consistent with the pre-collected measurements of the assumed active
mode. If this is not the case, then the assumed mode is not currently
active. We ensure the computational efficiency of this test by
providing conservative, but efficient, outer approximations to relax
the testing of non-emptiness of the intersection of convex sets.  Once
the active mode is determined, the controller switches to the
stabilization phase, which applies the corresponding stabilizing
feedback gain and monitors the evolution of the associated Lyapunov
function to detect switches. When a switch is detected, the controller
returns to the mode-detection phase.  We study the practical stability
of the closed loop of the unknown system under the proposed online
switched controller and show that, under average dwell- and
activation-time assumptions on the switching signal, it enjoys a 
{\color{black}practical} stability property. In particular, when
measurements are noiseless, the closed loop is asymptotically stable.
% the system does not switch while it is detecting the active
% mode. Moreover, we assume average dwell time (ADT) and average
% activation time (AAT) conditions for the mode detection phase. Under
% these conditions, the closed loop satisfies an input-to-state-like
% stability property.

We have presented preliminary results of this work in the
  conference article~\cite{JE-SL-SM-JC:22-cdc}, where we only
  considered the case of noiseless data and established practical
  stability of the closed-loop system. The main contribution of the
  present treatment is considering noisy measurements and solving a
  more general problem, with stronger stability results. Dealing with
  noise requires significant extensions to all the ingredients of the
  approach, a complete generalization of the results for the
  initialization step, and an entirely new set of techniques for mode
  detection. These, along with refinements in the choice of bounded
  inputs to excite the system, lead to stronger practical stability
  results, yielding in the noiseless case asymptotic stability of the
  closed-loop system. The stronger stability results obtained here are
  illustrated in an extensive set of novel simulation results.

% \emph{Organization:} The paper is structured as
% follows. Section~\ref{sec:notion} introduces necessary background and
% provides the problem statement.  Section~\ref{sec:initialization}
% describes the initialization step and introduces theoretical notions
% related to data informativity. These notions also play a key role in
% Section~\ref{sec:MD}, where we consider the problem of determining the
% currently active mode using online measurements.
% Section~\ref{sec:stabilization} discusses the stabilization phase of
% the controller and the criteria for detecting switches.  With all the
% pieces in place, we introduce the online switched controller design in
% Section~\ref{sec:osc}, where we also characterize its stabilization
% properties.  Section~\ref{sec:simulation} provides illustrative
% simulations and Section~\ref{sec:conclusion} gathers our concluding
% remarks.

\section{Problem formulation}\label{sec:notion}

Consider\footnote{Throughout the paper, we use the following
  notation. We denote by $\N$ and $\R$ the sets of non-negative
  integer and real numbers, respectively. We let $\R^{n\times m}$
  denote the space of $n\times m$ real matrices. For any
    $x\in\R^n$, $\Vert x\Vert$ denotes the standard 2-norm and for any
    $M\in\R^{n\times m}$, $\Vert M\Vert$ denotes the induced norm.
  For $P\in\R^{n\times n}$, $P\succeq 0$ (resp. $P\succ 0$) denotes
  that $P$ is positive semi-definite (resp. definite).} a
discrete-time switched linear system with $p$ modes of the form
\begin{equation}\label{def:sys}
  x(t+1)=\hat{A}_{\sigma(t)}x(t)+\hat{B}_{\sigma(t)}u(t) +w(t),
\end{equation}
where $x(t)\in\R^n$ is the state, $u(t)\in\R^m$ is the control input,
and $w(t)\in\mathbb{R}^n$ is a noise signal. We will assume that
  the signal $w$ is bounded, but make no assumptions regarding
  randomness. Here, $\sigma:\N\rightarrow\calP:=\{1,2,\cdots,p\}$ is
the switching signal. For all $i\in\calP$, the matrices $\hat{A}_i$
and $\hat{B}_i$ are in $\mathbb{R}^{n\times n}$ and
$\mathbb{R}^{n\times m}$, respectively.

We are interested in finding a controller on the basis of
  measurements such that the closed-loop system is \emph{practically
    exponentially stable}, i.e., there exist constants $c>0$,
  $\zeta\in(0,1)$ and $r\geq0$ (depending on a bound on the noise)
  such that its trajectories satisfy
  \begin{equation}\label{ISS-estimate_0}
    \norm{x(t)}\leq c \, \zeta^t\norm{x(0)}+r ,
  \end{equation}
  for all initial states $x(0)\in\R^n$ and all time $t\in\N$.

To be precise, we assume that the dimensions $n$, $m$ and the number
of modes $p$ are known, but that the precise dynamics of the modes are
not available for design, that is, for each mode $i\in\calP$, the
matrices $\hat{A}_i$ and $\hat{B}_i$ are unknown. To offset this lack
of knowledge, we have access to a finite set of measurements of the
state and input trajectories, $x(t),u(t)$, corresponding to an
  unknown but bounded noise signal $w(t)$. Specifically, we assume
that, before the online operation of the system, we perform an
\textit{initialization} step, where we obtain measurements of each of
the individual modes (but there is not necessarily enough data to
identify the dynamics of each separate mode).  After this, once the
system is running, we have access to \textit{online} measurements of
the currently active mode. Our goal is formalized as follows.

  \begin{problem}[Online switched controller design]\label{problem}
    Given initialization and online
    measurements, design a control law $u:\N\rightarrow\R^m$ such that
    the resulting interconnection {\color{black}with} \eqref{def:sys} is practically
    exponentially stable, cf~\eqref{ISS-estimate_0}. For each
    $t\in \N$, the control $u(t)$ only depends on (but not necessarily
    all) the following data:
    \begin{itemize}
    \item Current state $x(t)$;
    \item Measurements collected during the initialization step;
    \item Online measurements
      $x(s),u(s), s\in\{t-T,t-T+1,\ldots,t-1\}\cap\N$ for some
      $T\in\N\backslash\{0\}$.
    \end{itemize}
\end{problem}

{\color{black} In particular, this means that we aim to develop a
  controller which does not use the full, unbounded history of data
  collected from the system, and hence has bounded complexity in terms
  of computational and memory cost.}

{\color{black} We consider this problem in both the absence and
  presence of bounded noise. Moreover, we consider two different cases
  regarding the switching signal. At first, we consider the situation
  where the controller is aware of \textit{when} the systems switches
  modes, but not to which mode. After this, we explore the situation
  where the switching signal is completely unknown. In order to
  guarantee stabilization in this case, we require certain regularity
  assumptions on the switching signal. To be precise, we assume that a
  switch does not happen during the collection of online measurements,
  and that we know bounds on the average switching frequency and
  elapsed time between switches. }

In order to solve Problem~\ref{problem}, we employ the following
multi-pronged approach. Since unique models for each of the modes
cannot necessarily be determined, we employ the concept of
\textit{data informativity} to formulate conditions under which the
initialization measurements guarantee the existence of a stabilizing
feedback controller for each of the modes. This problem is solved
  in Section~\ref{sec:initialization}. Based on this, we describe a
  switched controller which operates in two phases. In the
  \textit{mode detection phase}, the controller selects inputs which
  allow us to determine the active mode. 
%  In general this can be shown
%  to require less measurements than fully identifying the system.
%In fact, in the absence of noise, a bound on the required number of steps can be guaranteed by a suitable choice of inputs. 
  The technical details of this phase are derived in
  Section~\ref{sec:MD}. After this, and once the active mode is
  identified, the controller switches to the \textit{stabilization
    phase}, where the controller found in the initialization step
  corresponding to the mode is applied. Describing this phase is done
  in Section~\ref{sec:stabilization}, and the remainder of the paper
  deals with analyzing the closed-loop behavior resulting from the
  proposed controller.

\section{Initialization step}\label{sec:initialization}

We begin our analysis with the initialization step, where we consider
the problem of finding stabilizing controllers from pre-collected
measurements of the system. To formalize the notion of informativity,
we require some notation. For simplicity of exposition, we first
consider a single linear system, then shift our focus to switched
systems.

Consider measurements of the state and input signals $x$ and $u$ of a
system
\[
  x(t+1) = \hat{A}x(t) + \hat{B}u(t) +w(t),
\] 
on the time interval $\{0,\ldots, T\}$. We define matrices
\begin{align*}
  X 	&:=\begin{bmatrix}x(0) & \cdots & x(T) \end{bmatrix},\\ 
  X_- 	&:=\begin{bmatrix}x(0) & \cdots & x(T-1) \end{bmatrix},\\ 
  X_+ 	&:=\begin{bmatrix}x(1) & \cdots & x(T) \end{bmatrix},\\ 
  U_-	&:=\begin{bmatrix}u(0) & \cdots & u(T-1)\end{bmatrix},\\
  W_-	&:=\begin{bmatrix}w(0) & \cdots & w(T-1)\end{bmatrix}.
\end{align*}
We assume that the state and input measurements $(U_-,X)$ are known,
but that the noise signal is unknown. {\color{black} However, the noise
  satisfies the following assumption.
  
  \begin{assumption}[The noise model]
    The matrix $W_-$ collecting the noise signal is such that
    \begin{equation}\label{eq:noise model}
      \begin{bmatrix}
        I_n
        \\
        W_-^\top
      \end{bmatrix}^\top
      \begin{bmatrix}
        \Pi_{11} &
                   \Pi_{12}
        \\ \Pi_{21} &
                      \Pi_{22}
      \end{bmatrix}
      \begin{bmatrix}
        I_n
        \\
        W_-^\top
      \end{bmatrix}
      \succeq 0,
    \end{equation}
    where $\Pi_{11}\in\mathbb{R}^{n\times n}$ and
    $\Pi_{22}\in\mathbb{R}^{T\times T}$ are symmetric, and
    $\Pi_{12}=\Pi_{21}^\top\in\mathbb{R}^{n\times T}$.  Moreover,
    throughout the paper, we assume $\Pi_{22}\prec 0$ and
    $\Pi_{11}- \Pi_{12}\Pi_{22}^{-1}\Pi_{21}\succeq 0$.
  \end{assumption}

  These latter assumptions guarantee that the set of matrices $W_-$
  for which \eqref{eq:noise model} holds is nonempty and
  bounded~\cite[Theorem 3.2]{HJVW-MKC-JE-HLT:22}.}

We note that such sets can naturally be used as confidence intervals
in the setting of Gaussian noise (see
\cite[Section~5.4]{HJVW-MKC-JE-HLT:22}). Moreover, these sets arise
when we bound the energy of the signal. To see this, note that
\[
  \sum_{t=0}^{T-1} w(t)w(t)^\top = W_-W_-^\top \preceq Q,
\]
is equivalent to a noise model \eqref{eq:noise model} with
\begin{equation}\label{eq:noise model restricted}
  \Pi
  := \begin{bmatrix}
       \Pi_{11} & \Pi_{12}
       \\
       \Pi_{12}^\top & \Pi_{22}
     \end{bmatrix}
     =
     \begin{bmatrix}
       Q & 0
       \\
       0 & -I_{T}
     \end{bmatrix}.
   \end{equation} 
   {\color{black} This latter form can also be employed to handle
     sample bounded noise. To be precise, one can derive that:  
     \[
       \norm{w(t)}_2 \leq q \quad \forall t = 1,\ldots, T \implies
       W_-W_-^\top \preceq q^2T I_n.
     \]
    However, this is conservative, given that the reverse implication
   does not hold.}
   
%     Noise models of this form are very versatile. Depending on the choice
%of $Q$, we can model for instance, the assumption of a signal-to-noise
%ratio by taking $Q= \gamma X_-X_-^\top$. The noise model also captures
%the case of exact measurements using $Q=0$.

%\begin{remark}[Noiseless measurements]\label{rem:noiseless}
%  A special case of the previous is the case where $\Pi_{22}=-I_T$,
%  $\Pi_{12}=0$, and $\Pi_{11}=0$. In this case, the matrix $W_-$
%  satisfies \eqref{eq:noise model} if and only if $W_-=0$, that is,
%  the measurements are without noise. % As we show in the remainder of
%  % the paper, the situation without noise reveals a number of important
%  % differences as compared the situation with noise.
%  \oprocend
%\end{remark}

We can now define the set of systems \textit{consistent} with the
measurements as
\begin{align*}%\label{def:Sigma}
	\Sigma(U_-,X)& :=\ \big\{(A,B):
	\ \exists W_-\in\R^{n\times T}
	\\&\mbox{ s.t. }X_+ = AX_- +  BU_- +W_- \,\mbox{ and }
      \eqref{eq:noise model} \mbox{ hold.}\big\} 
\end{align*}
Clearly, since the measurements satisfy the noise model and are
collected from the true system, we know that $(\hat{A},\hat{B})$ is
contained in the set $\Sigma(U_-,X)$. If we now define the matrix
\begin{equation}\label{def:N}
  N := \begin{bmatrix} I_n\!\!
    &X_+\\0\!\!&-X_-\\0\!\!&-U_-
  \end{bmatrix}\!\!\begin{bmatrix} \Pi_{11} & \!\!\Pi_{12} \\ \Pi_{21}
    & \!\!\Pi_{22} \end{bmatrix}\!\!\begin{bmatrix} I_n
    \!\!&X_+\\0\!\!&-X_-\\0\!\!&-U_-
  \end{bmatrix}^\top,
\end{equation}
it is straightforward to conclude that the set $\Sigma(U_-,X)$ can be
equivalently represented as 
\begin{equation}\label{eq:Sigma as QMI}
  \Sigma(U_-,X) =\bigg \lbrace
    (A,B) \mid \begin{bmatrix} I_n \\ A^\top \\
      B^\top \end{bmatrix}^\top N \begin{bmatrix} I_n \\
      A^\top\\B^\top \end{bmatrix}\succeq 0 \bigg\rbrace.
\end{equation}
We are interested in characterizing properties of the true system
based on the measurements. However, the set $\Sigma(U_-,X)$ might
contain other systems in addition to the true one. This means, for
instance, that we can only conclude that a feedback gain $K$
stabilizes the true system if this gain stabilizes any system whose
system matrices are in $\Sigma(U_-,X)$. We are interested in
determining when the available data is informative enough to allow us
to accomplish this.

\begin{definition}[Informativity for uniform
  stabilization]\label{def:informativity for uniform stabilization}
  The data $(U_-,X)$ is \textit{informative for uniform stabilization
    by state feedback with decay rate} $\lambda\in(0,1)$ if there
  exist $K\in\R^{m\times n}$ and $P\in\R^{n\times n}, P\succ 0$ such
  that
  \begin{equation}\label{eqn:informativity for uniform stabilization}
    (A+BK)^\top P (A+BK)\prec\lambda^2 P , \quad\forall (A,B)\in\Sigma(U_-,X).
  \end{equation}
\end{definition}

Let $V(x):=\sqrt{x^\top Px}$ and consider the closed loop of any
system consistent with the data and the controller $u=Kx$. Then, if
the data is informative,
\begin{align}\label{eq:decay-rate}
  V(x(t+1))
  % & =x(t+1)^\top P x(t+1)
  % \\
  % & = x(t)^\top (A+BK)^\top P(A+BK) x(t)
  % \\
    & \leq \lambda \sqrt{x(t)^\top P x(t)}=\lambda V(x(t)),
\end{align}	
for all $t \in \N$. This means that, even if the matrices $(A,B)$ can
not be identified uniquely from the measurements, the feedback gain
$K$ stabilizes the system with Lyapunov function $V$ and decay
rate $\lambda$.

To determine whether the data $(U_-,X)$ is informative, we can exploit
the fact that \eqref{eq:Sigma as QMI} and \eqref{eqn:informativity for
  uniform stabilization} are quadratic matrix inequalities in $A$ and
$B$. The following result
extends~\cite[Theorem~5.1.(a)]{HJVW-MKC-JE-HLT:22}, which essentially
considers the special case $\lambda=1$, to reduce the problem to that
of finding matrices $K$ and $P$ that satisfy a \emph{linear matrix
  inequality} (LMI).

\begin{theorem}[Conditions for informativity for uniform
  stabilization]\label{thm:common_K}
  The data $(U_-,X)$ is informative for uniform stabilization by state
  feedback with decay rate $\lambda$ if and only if there exist
  $Q\in\R^{n\times n}$, with $Q\succ 0$, $L\in\R^{m\times n}$ and $\beta>0$ such that
  \begin{equation}\label{eqn:common_K_LMI}
    \small \!\begin{bmatrix}
      \!\lambda^2 Q\!-\!\beta I_n\!\!\!\!\!& 0 & 0 &0\\0&0&0 &Q\\0&0&0&L\\0&Q&\! L^\top\!\!\! &Q
    \end{bmatrix}-\begin{bmatrix} I_n\!\! &X_+\\0\!\!&\!-X_-\\0\!\!&\!-U_-\\0\!\!&0
    \end{bmatrix}\!\!\!\begin{bmatrix} \Pi_{11} & \!\!\Pi_{12} \\ \Pi_{21} & \!\!\Pi_{22} \end{bmatrix}\!\!\!\begin{bmatrix} I_n \!\!&X_+\\0\!\!&\!-X_-\\0\!\!&\!-U_-\\0\!\!&0
\end{bmatrix}^{\!\!\top}\!\!\!\!\succeq 0.
\normalsize
  \end{equation}
  Moreover, the matrices $K:=LQ^{-1}$ and $P:=Q^{-1}$
  satisfy~\eqref{eqn:informativity for uniform stabilization}.
\end{theorem}

% Note that the feasibility of this inequality is a necessary and
% sufficient condition for the data $(U_-,X)$ to be informative for
% uniform stabilization with decay rate $\lambda$ by state feedback.
Theorem~\ref{thm:common_K} can be proven with the same argument as
  in the proof of \cite[Theorem~5.1.(a)]{HJVW-MKC-JE-HLT:22}, making
  minor adjustments regarding the stability notion and related steps.
The characterization given in this result provides the backbone of the
initialization step for our controller design.  As noted in the
problem formulation, we have access to measurements of each of the
modes in the initialization step.  We denote the measurements
corresponding to mode $i\in\calP$ by $(U_-^i,X^i)$. We consider the
corresponding matrices $U_-^i$, $X^i$ and a noise model given
by~$\Pi^i$. Since we are interested in stabilizing the unknown
switched system regardless of the switching signal, this implies that,
each of the modes must be stabilizable. This motivates the following
assumption.

\begin{assumption}[Initialization step]\label{ass:ass_on_data-I}
  Let $\lambda\in (0,1)$.  For each mode $i\in\calP$, the data
  $(U_-^i,X^i)$ is informative for uniform stabilization by state
  feedback with decay rate $\lambda$.
\end{assumption}

 \begin{remark}[Associating measurements to modes in the
    initialization step]
    Implicitly, Assumption~\ref{ass:ass_on_data-I} requires that we
    know to which mode a given set of measurements corresponds.  If
    this were not the case, one could devise a clustering method in
    order to assign measurements collected in an unstructured set to
    the separate modes. For a number of realistic scenarios, this is
    indeed an important first step but we do not pursue it here
    because of the many extra technicalities required by clustering.
    \oprocend
\end{remark}

We denote by $K_i$ any feedback corresponding to $(U_-^i,X^i)$
obtained from Theorem~\ref{thm:common_K}. The corresponding Lyapunov
matrix is denoted by $P_i$. This means that we have (potentially)
different feedback and Lyapunov matrices for each mode, but that the
decay rate is uniform for all modes $i\in\calP$. Note that the choice
of a common decay rate across the modes does not introduce
conservatism, since if \eqref{eqn:informativity for uniform
  stabilization} holds for $\lambda$, it also holds for any
$\bar{\lambda}$ with $\lambda\leq \bar{\lambda}<1$.

\section{Mode detection phase}\label{sec:MD}

After the initialization step, we consider the situation where the
unknown system is in operation and additional \textit{online}
measurements, denoted by
$(U^{\textnormal{on}}_-,X^{\textnormal{on}})$, are collected. As before, we let
$\Sigma(U^{\textnormal{on}}_-, X^{\textnormal{on}})$ denote the set of
all systems compatible with the online measurements. To simplify the
exposition, we introduce for each $i\in\calP$ the following shorthand
notation,
\[
  \Sigma^i := \Sigma(U^i_-,X^i), \quad \Sigma^{\textnormal{on}} :=
  \Sigma(U^{\textnormal{on}}_-, X^{\textnormal{on}}).
\]
%We assume that the systems dwells in a single mode, and that we
%collect the online measurements sequentially. As such, after $\ell$
%time instances, we can write
%\[
%  \Sigma^{\textnormal{on}} = \bigcap_{t=0}^{\ell-1}
%  \Sigma^{\textnormal{on}}_t,
%\]
%where $\Sigma^{\textnormal{on}}_t$ is the set of systems consistent
%with the measurements collected at time instance $t$.

Recall that we aim at stabilizing the switched system.  Using the
  previously derived results, we could simply test whether the online
  measurements $(U^{\textnormal{on}}_-, X^{\textnormal{on}})$ are
  informative for uniform stabilization with decay rate~$\lambda$.
However, the initialization data gives us additional information that
can be exploited: we know that the true system is contained in
$\Sigma^i\cap\Sigma^{\textnormal{on}}$ for at least one $i\in\calP$.
If we could uniquely determine which of the $p$ different systems has
generated the online measurements
$(U_-^{\textnormal{on}},X^{\textnormal{on}})$, we can then apply the
stabilizing controller found in the initialization step.  The notion
of data compatibility plays a key role in achieving this goal.

\begin{definition}[Data compatibility]
  The data pairs $(U_-^i,X^i)$ and $(U_-^j,X^j)$ are
  \textit{compatible} if there exists a system that is consistent with
  both, that is, $\Sigma^i\cap\Sigma^j\neq\emptyset$.
\end{definition}

\begin{figure}
  \centering
  \begin{tikzpicture}[scale=0.85]
    % left hand
    % \scope
    % \clip (-2,-2) rectangle (2,2)
    % (1,0) circle (1);
    % \fill (0,0) circle (1);
    % \endscope
    % % right hand
    % \scope
    % \clip (-2,-2) rectangle (2,2)
    % (0,0) circle (1);
    % \fill (1,0) circle (1);
    % \endscope
    % outline
    \filldraw[color=black, fill=black, fill opacity=0.2] (2,3) circle (1.2) (2,3)  node [opacity=1] {$\Sigma^1$}
    (5,2) circle (1.2) (5,2)  node [opacity=1] {$\Sigma^2$}
    (7.5,4) circle (1.2) (7.5,4)  node [opacity=1] {$\Sigma^3$};
    \filldraw[blue, opacity=0.25]  (5,3.5) circle (2);
    \filldraw[blue, opacity=0.25]  (5,3.5) circle (1.5);
    \filldraw[blue, opacity=0.25]  (5,3.5) circle (1) (5,3.5) node [black, opacity=1] {$\Sigma^{\textnormal{on}}$};
    \draw (0,0.5) rectangle (10,6) (1.5,5.5) node {$\R^{n\times n}\times \R^{n\times m}$};
  \end{tikzpicture}
  \caption{Graphical interpretation of the mode detection
    scheme. Initially, $\Sigma^{\textnormal{on}}$ intersects the sets
    corresponding to three different modes.  As more data are
    collected, $\Sigma^{\textnormal{on}}$ decreases in size (cf.
    darker blue disks). When enough data are available,
    $\Sigma^{\textnormal{on}}$ eventually becomes compatible only with
    mode 2. Note that we do not need $\Sigma^{\textnormal{on}}$ to
    shrink into a singleton (system identification) before reaching
    unique compatibility.}\label{fig:mode_detection}
\end{figure}
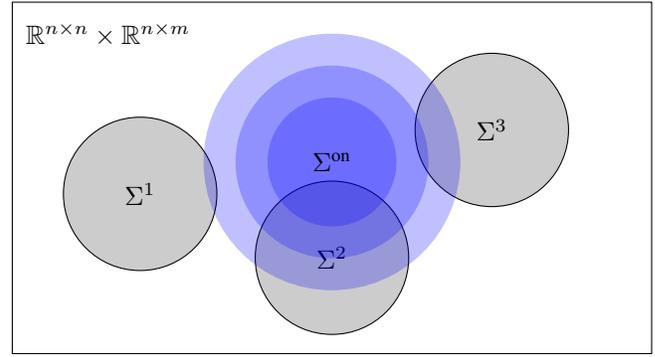

% \todoinr{Later: change the algorithm such that either (1) the
% online data is compatible with only one mode \textit{or} (2) we have
% $\Sigma^{\textrm{on}}\subseteq \Sigma^i$ for some $i$. In the latter
% case, we can simply apply the controller corresponding to mode
% $i$. Of course, denoting the obvious things with $N^i$ and $M^i$, it
% makes little sense to check whether $Z(N^{on})\subseteq Z(N^i)$ over
% just checking $Z(N^{on})\subseteq Z(M^i)$.As such, we can check
% whether $\Sigma^{on}$ can be stabilized with any of the $K_i's$.}

We are interested in determining the active mode of the system from
the online measurements.

\begin{definition}[Informativity for mode detection]
  Given initialization data $\{(U_-^i,X^i)\}_{i\in\calP}$, the online
  measurements $(U_-^{\textnormal{on}},X^{\textnormal{on}})$ are
  \emph{informative for mode detection} if
  $(U_-^{\textnormal{on}},X^{\textnormal{on}})$ and $(U_-^i,X^i)$ are
  compatible for exactly one $i\in\calP$.
\end{definition}

As a prerequisite for the mode detection phase to be viable, we
  require that the initialization data are pairwise
  incompatible. Intuitively, this requires that the true mode dynamics
  are `sufficiently' distinct, in the sense that the uncertainty
  induced by the respective measurements is small enough to
  distinguish the modes.

\begin{assumption}[Initialization step
  --cont'd]\label{ass:ass_on_data-II}
  The data $\{(U_-^i,X^i)\}_{i\in\calP}$ are such that $(U^i_-,X^i)$
  and $(U^j_-,X^j)$ are incompatible for each pair $i\neq j \in\calP$.
\end{assumption}

Since the online measurements are generated by precisely one of the
modes $i\in\calP$, there must be at least one $i\in\calP$ such that
$(U_-^{\textnormal{on}},X^{\textnormal{on}})$ and $(U_-^i,X^i)$ are
compatible. By assuming that the initial data are pairwise
incompatible, we can conclude that once $\Sigma^{\textnormal{on}}$
becomes sufficiently ``small'', the mode $i$ corresponding to the data
$(U_-^i,X^i)$ which are compatible with
$(U_-^{\textnormal{on}},X^{\textnormal{on}})$ must be
unique. Figure~\ref{fig:mode_detection} provides a graphical
illustration. 
%We also remark that from the interpretation depicted in Figure~\ref{fig:mode_detection}, The complexity of our mode detection scales linearly with respect to the number of modes under Assumption~\ref{ass:ass_on_data-II}.

\subsection{Mode detection for noiseless data}
First, we consider the system without noise, that is, $w(t)=0$. The
following result characterizes compatibility in this case.

\begin{lemma}[Conditions for noiseless data compatibility]\label{lem:exact dis}
  Noiseless data $(U_-^1,X^1)$ and $(U_-^2,X^2)$ are compatible if and
  only if
  \begin{equation}\label{eq:kernel}
    \ker
    \begin{bmatrix}
      X^1_-&X^2_-\\U^1_-&U^2_-
    \end{bmatrix}
    \subseteq  
    \ker \begin{bmatrix}X^1_+&X^2_+\end{bmatrix}. 
  \end{equation}
\end{lemma}
\begin{proof}
  Since $W_-^1=W_-^2=0$, the data are compatible if and only if there
  exists $A$ and $B$ such that: $X^1_+ = AX^1_-+BU_-^1$ and
  $X^2_+ = AX^2_-+BU_-^2$. Equivalently, we have
  \[
    \begin{bmatrix}X^1_+&X^2_+\end{bmatrix} = \begin{bmatrix} A&
      B \end{bmatrix} \begin{bmatrix}
      X^1_-&X^2_-\\U^1_-&U^2_- \end{bmatrix}.
  \]
  The existence of such $\begin{bmatrix} A&B \end{bmatrix}$ is
  equivalent to \eqref{eq:kernel}.
\end{proof}

This result provides a test to check whether the online measurements
are compatible with the initialization data. At each step, the idea is
to evaluate whether the online data are compatible with precisely one
mode of the system. This iterative procedure raises the question of
how to select the input to the unknown system appropriately at each
step.  In other words, we are interested in \textit{generating} online
data such that, after a bounded number of steps,
$\Sigma(U_-^{\textnormal{on}},X^{\textnormal{on}})$ becomes small
enough for the data to be informative for mode detection. Formally,
the problem is to find a time horizon $T^{\textnormal{on}}$ and inputs
$u^{\textnormal{on}}(0),\ldots,
u^{\textnormal{on}}(T^{\textnormal{on}}-1)$ such that the
corresponding online data
$(U_-^{\textnormal{on}},X^{\textnormal{on}})$ are informative for mode
detection.

To obtain such inputs, we adapt the experiment design method
of~\cite{HJVW:21}. When each of the modes $(\hat{A}_i,\hat{B}_i)$ of the system
are controllable, \cite[Theorem~1]{HJVW:21} gives a construction for
inputs $u^{\textnormal{on}}(0),\ldots, u^{\textnormal{on}}(n+m-1)$
such that the corresponding set
$\Sigma(U_-^{\textnormal{on}},X^{\textnormal{on}})$ is a
singleton. Assume
\[
  \rank \begin{bmatrix}
    X^{\textnormal{on}}_-\\U^{\textnormal{on}}_- \end{bmatrix}\neq
  n+m.
\]
Now, either $x(t)\not\in\im X^{\textnormal{on}}_-$, in which case,
take $u(t)$ equal to $0$. In the other case, \cite[Theorem~1]{HJVW:21}
shows that there exist $\eta,\xi$ such that $\eta\neq 0$ and
\[
  \begin{bmatrix} \xi\\\eta \end{bmatrix}\in\ker\begin{bmatrix}
    X^{\textnormal{on}}_-\\U^{\textnormal{on}}_- \end{bmatrix}^\top .
\] 
Now, by taking any $u(t)\in\R^m$ such that
$\xi^\top x(t)+\eta^\top u(t)\neq 0$, it can be shown that
\[
  \rank \begin{bmatrix}
    X^{\textnormal{on}}_-\\U^{\textnormal{on}}_- \end{bmatrix}<
  \rank \begin{bmatrix} X^{\textnormal{on}}_- & x(t)
    \\U^{\textnormal{on}}_- & u(t) \end{bmatrix}.
\]
In order to {\color{black}contain} the destabilizing effect of this excitation,
we choose $u(t)$ such that $\norm{u(t)}\leq c\norm{x(t)}$ for some $c>0$.

Clearly, after $n+m$ repetitions, this procedure results in a data
matrix which has full row rank. This implies that
$\Sigma(U_-^{\textnormal{on}},X^{\textnormal{on}})$ is a singleton. As
such, under Assumption~\ref{ass:ass_on_data-II}, the measurements up
to $T^{\textnormal{on}}=n+m$ are guaranteed to be informative for mode
detection. This provides a worst-case bound as, in general, mode
  detection is achieved with far fewer measurements than the number
  required to uniquely determine the system on the basis of online
  data alone.

\begin{algorithm}[htb!]
  \caption{Mode detection for noiseless
    data}\label{alg:mode_detection}
  % \caption{Each time instance of the mode detection algorithm for
  % noiseless
  % data}\label{alg:mode_detection}
  \algorithmicrequire
  $\calP_{\textrm{match}},\{U^i_-,X^i\}_{i\in\calP},U^{\textnormal{on}}_-,X^{\textnormal{on}},c$
  \\
  \algorithmicensure $\calP_{\textrm{match}},U^{\textnormal{on}}_-,X^{\textnormal{on}}
%  ,f_{\textup{done}}
$
  \begin{algorithmic}[1]
    	\If{$U_-^{\textnormal{on}}\neq []$ and $x(t)\in\im X^{\textnormal{on}}_-\setminus \{0\}$ }\label{step:mode_detection_start} 
	   		\State Pick $\begin{bmatrix}\xi\\\eta\end{bmatrix}\in\ker\begin{bmatrix} X^{\textnormal{on}}_-\\U^{\textnormal{on}}_-
    	\end{bmatrix}^\top$ with $\eta\neq 0$\label{step:mode_detection_critical}
    		\State Let $u(t)\in\R^m$ be such that \label{step:mode_detection_bound}
    		\[\norm{u(t)}\leq c\norm{x(t)} \textrm{ and } \xi^\top x(t)+\eta^\top u(t)\neq 0\] \LineComment{Choose the next input}
    	\Else
    		\State $u(t)\gets 0$
    	\EndIf\label{step:mode_detection_end}
 		\State Get the next state $x(t+1)$
 		
  		\State $U_-^{\textnormal{on}}\gets \begin{bmatrix}U_-^{\textnormal{on}}&u(t) \end{bmatrix}$                         
		\State $X^{\textnormal{on}}\gets \begin{bmatrix} X^{\textnormal{on}}& x(t+1)
\end{bmatrix}$ \Comment{Append the online data}
  \If{$x(t+1)=0$} \label{step:x0 start}
  \Comment{System stable for any feedback}
  \State $\calP_{\textnormal{match}} = \{1\}$     \label{step:x0 end}
  \Else
  \For{$i\in\calP_{\textrm{match}}$}  
    \If{the inclusion \eqref{eq:kernel} is violated}  \LineComment{Data are incompatible with mode $i$}
    \State $\calP_{\textrm{match}}=\calP_{\textrm{match}}\backslash\{i\}$ \Comment{Eliminate mode $i$}
    \EndIf
    \EndFor 
    \EndIf
%    \If{$|\calP_{\textrm{match}}|=1$}\LineComment{Mode is detected if only one matches}
%    \State $f_{\textup{done}}=1$
%    \Else
%    \State $f_{\textup{done}}=0$
%    \EndIf
  \end{algorithmic}
\end{algorithm}
Algorithm~\ref{alg:mode_detection} formalizes the mode detection
procedure for noiseless data.
% This algorithm requires as input the
% aforementioned data and a list of possible modes
% $\calP_{\textnormal{match}}$. In practice we start by taking
% $\calP_{\textnormal{match}}=\calP$, that is, we assume that any mode
% could be active. Then each iteration of
% Algorithm~\ref{alg:mode_detection} generates a new online measurement
% by choosing a bounded input consistent with the aforementioned
% experiment design method of~\cite{HJVW:21} in Step
% \ref{step:mode_detection_start} to Step \ref{step:mode_detection_end}.
% After applying this input and measuring the resulting state, we update
% the data matrices. 
Note that if, at any point of the algorithm, the state satisfies
$x(t)=0$ (step \ref{step:x0 start}), then choosing $u(t)=0$ stabilizes any linear
system, and therefore specifically all modes of the system. As such,
we can apply any of the feedback gains $K_i$. We enforce this by
setting $\calP_{\textnormal{match}}$ equal to any of the
modes. Otherwise, the algorithm checks whether \eqref{eq:kernel} is
violated for each of the modes remaining in
$\calP_{\textnormal{match}}$, in which case the corresponding mode is
discarded.  If $|\calP_{\textrm{match}}|\neq 1$, the algorithm should
be repeated and another online measurement collected.  If
$|\calP_{\textrm{match}}|=1$, the active mode is identified.
% Lastly, the variable $f_{\textup{done}}$ is toggled to 1 if the
% unique active mode is detected. If instead $f_{\textup{done}}$
% remains equal to 0, this algorithm is ran again.

\begin{corollary}[Algorithm~\ref{alg:mode_detection} terminates in a
  finite number of repetitions]\label{cor:at-most-n+m}
  Suppose that for each mode $i\in\calP$ the matrix pair
  $(\hat{A}_i,\hat{B}_i)$ is controllable and that
  Assumptions~\ref{ass:ass_on_data-I} and~\ref{ass:ass_on_data-II}
  hold.  Let $\calP_{\textnormal{match}}=\calP$ and
  $X^{\textnormal{on}} = [x(0)]$.  Execute Algorithm~1 iteratively,
  updating $\calP_{\textrm{match}}$ and
  $(U^{\textnormal{on}}_-,X^{\textnormal{on}})$ at every step.  Then
  $|\calP_{\textrm{match}}|=1$ after at most $n+m$
  iterations. Moreover, the online data
  $(U^{\textnormal{on}}_-,X^{\textnormal{on}})$ are either informative
  for mode detection or such that $x(T^{\textnormal{on}})=0$.
\end{corollary}

\subsection{Mode detection for noisy data}
As in the noiseless case, testing whether
$\Sigma^i\cap\Sigma^{\textnormal{on}}\neq \emptyset$ amounts to
checking non-emptiness of the intersection of two convex sets. This
needs to be done online, during the collection of measurements, and
hence requires to be resolved in time with the evolution of the
system. In the noiseless case, the fact that the convex sets under
consideration were affine made solving the problem online feasible by
employing Lemma~\ref{lem:exact dis}. However, this is no longer the
case in the presence of noise. Therefore, to be able to test for
(in)compatibility in an online fashion, here we develop a number of
conservative, but more computationally efficient, methods.  Our
exposition first shows how to over-approximate the sets of compatible
systems by spheres. This allows us to provide simple tests for
incompatibility after a single measurement.  After this, we also
propose methods to deal with sequential measurements.

\paragraph*{Outer approximation of set of consistent systems on the
  basis of measurements} For simplicity of exposition, we first
  deal with a single true linear system and a single set of
  measurements $(U_-,X)$.  Moreover, to further ease the notation, we
  assume in the remainder of the paper that the noise models are given
  in the form of a bound on the energy of the noise, that is,
  $W_-W_-^\top \preceq Q$, with $Q=Q^\top\succeq 0$, cf. \eqref{eq:noise
    model restricted}. The results presented below can be adapted to
  the more general case in a straightforward manner.

Assuming that $\begin{bmatrix} X_-\\U_-\end{bmatrix}$ has full row rank, we
define the \textit{center} $Z$ of the set $\Sigma= \Sigma(U_-,X)$ of
consistent systems, given as in \eqref{eq:Sigma as QMI}, by
\begin{equation}\label{eq:centre}
  Z := X_+\begin{bmatrix} X_- \\
    U_- \end{bmatrix}^\dagger,
\end{equation}
where $M^\dagger$ denotes the Moore-Penrose inverse of a matrix
$M$. Let $\lambda_{\min}(M)$ and $\lambda_{\max}(M)$ denote the
smallest and largest eigenvalue of $M$ respectively, and define the
\textit{radius} $r$ of $\Sigma$ as
\begin{equation}\label{eq:radius}
  r:=
  \sqrt{\frac{\lambda_{\max}(Q)}{\lambda_{\min}\left(\begin{bmatrix}
          X_-\\U_-\end{bmatrix}\begin{bmatrix}
          X_-\\U_-\end{bmatrix}^\top\right)}}.
\end{equation} 
The following result provides an outer approximation of the
set~$\Sigma$.

\begin{lemma}[Outer approximation of set of consistent
  systems]\label{lem:bounded}
  Suppose that $\Pi$ is given by~\eqref{eq:noise model restricted} and
  $\begin{bmatrix} X_-\\U_-\end{bmatrix}$ has full row rank. Then
  $\Sigma\subseteq B^r(Z) = \left\lbrace (A,B) \mid
    \norm{\begin{bmatrix} A& B\end{bmatrix}-Z} \leq r\right\rbrace$,
  where $Z$ is the center~\eqref{eq:centre} and $r$ is the
  radius~\eqref{eq:radius} of~$\Sigma$.
\end{lemma}
\begin{proof}
  By definition $(A,B)\in\Sigma$ if and only if
  \[
    \left(X_+ - \begin{bmatrix} A& B\end{bmatrix}\begin{bmatrix}
        X_-\\U_-\end{bmatrix}\right) \left(X_+ - \begin{bmatrix} A&
        B\end{bmatrix}\begin{bmatrix}
        X_-\\U_-\end{bmatrix}\right)^\top \preceq Q.
  \] 
  This implies that 
  \[
    \left(\begin{bmatrix} A& B\end{bmatrix}-Z \right) \begin{bmatrix}
      X_-\\U_-\end{bmatrix}\begin{bmatrix} X_-\\U_-\end{bmatrix}^\top
    \left(\begin{bmatrix} A& B\end{bmatrix}-Z \right)^\top \preceq Q.
  \] 
  Since $\begin{bmatrix} X_-\\U_-\end{bmatrix}$ has full row rank, the
  matrix
  $\begin{bmatrix} X_-\\U_-\end{bmatrix}\begin{bmatrix}
    X_-\\U_-\end{bmatrix}^\top$ is positive definite. Moreover, for
  any matrix $M\in\mathbb{R}^{n\times n}$ such that $M\succeq 0$, we have
  $\lambda_{\min}(M)I_n \preceq M \preceq \lambda_{\max}(M)I_n$. This implies
  that
  \[
    \left(\begin{bmatrix} A& B\end{bmatrix}-Z \right)
    \left(\begin{bmatrix} A& B\end{bmatrix}-Z \right)^\top \preceq r^2 I_n,
  \]
  and hence we $\norm{\begin{bmatrix} A& B\end{bmatrix}-Z}^2 \leq r^2$
  for any $(A,B)\in\Sigma $.
\end{proof}

Conversely, if $\begin{bmatrix} X_-\\U_-\end{bmatrix}$ does not have
full row rank, then the corresponding set $\Sigma$ is unbounded, and
hence there does not exist any $r$ and $Z$ such
that~$\Sigma\subseteq B^r(Z)$.

\paragraph*{Conservative tests for incompatibility using outer approximations}

Here we describe how to leverage the outer approximation on the set of
systems consistent with some given measurements to determine data
incompatibility.  Given the initialization data, we define the
distance between $\Sigma^i$ and $\Sigma^j$ for $i,j\in\calP$ by
\[
  d_{ij} := \min_{\substack{(A_i,B_i)\in\Sigma^i \\
      (A_j,B_j)\in\Sigma^j}} \norm{\begin{bmatrix} A_i-A_j &
      B_i-B_j \end{bmatrix} } .
\]
Since the sets $\Sigma^i$ are closed by definition, the measurements
$(U^i_-,X^i)$ and $(U^j_-,X^j)$ are incompatible if and only if
$d_{ij}>0$. This distance can be determined from the initialization
data and computed in the initialization step.  We can give an
alternative, more computationally efficient test under the additional
assumption that the matrices
$\begin{bmatrix} X^i_-\\U^i_-\end{bmatrix}$ have full row rank. In
this case, we can use Lemma~\ref{lem:bounded} to efficiently bound the
distance between $\Sigma^i$ and $\Sigma^j$ from below. We denote by
$Z_i$ and $r_i$, respectively, the center and radius of $\Sigma^i$,
for each $i \in \calP$.  Then, we have
\begin{equation}
  d_{ij} \geq \norm{Z_i-Z_j}^2 -r_i-r_j. 
\end{equation} 
This can be used to formulate an efficient method of verifying
Assumption~\ref{ass:ass_on_data-II} as follows.

\begin{corollary}[Ensuring Assumption~\ref{ass:ass_on_data-II} holds]
	\label{cor:ens ass 3}
  Suppose that $\Sigma^i\subseteq B^{r_i}(Z_i)$ for all $i\in
  \calP$. If $ \norm{Z_i-Z_j}^2 >r_i+r_j $, for all $i\neq j$, then
  Assumption~\ref{ass:ass_on_data-II} holds.
\end{corollary}

The same idea can be used for the case of online measurements
$(U_-^{\textnormal{on}},X^{\textnormal{on}})$. Recall that these are
collected sequentially, giving rise to
$\Sigma^{\textnormal{on}} = \bigcap_{t=0}^{T^{\textnormal{on}}-1}
\Sigma^{\textnormal{on}}_t$.  Since the ultimate goal of checking for
incompatibility is to determine the active mode of the switched
system, after collecting each measurement, we face the dichotomy of
already using the information or wait for the additional one provided
by subsequent measurements. As a first step towards resolving this
dichotomy, we develop a computationally efficient test to check for
compatibility with a single set~$\Sigma^{\textnormal{on}}_t$
determined by a \emph{single} measurement.
% One approach would be to collect a large number of
% measurements and then test for (in)compatibility. Given that we are
% interested in determining the active mode in a small number of
% steps, we instead choose to focus on an approach where we collect a
% \textit{single} measurement at each time instance and then perform
% a computationally efficient test.

In the single measurement case, the noise model~\eqref{eq:noise model
  restricted} takes the form $w(t)w(t)^\top\preceq q^2I_n$, or
equivalently, $\norm{w(t)}\leq q$. Then,
\begin{equation}\label{eq:def Sigonj}
  \Sigma^{\textnormal{on}}_t = \{
  (A,B) \mid \norm{ \begin{bmatrix} A &
      B \end{bmatrix}\left(\!\begin{smallmatrix} x(t)\\
        u(t) \end{smallmatrix}\!\right) -x(t+1)}\leq q \}.
\end{equation}
The following result provides a sufficient condition for incompatibility
with~$\Sigma^{\textnormal{on}}_t$.

\begin{lemma}[Scalar test for incompatibility with single online
  measurement] \label{lem:dist for compatibility}
  Let $\Sigma\subseteq B^r(Z)$ for some $r$ and $Z$. For 
  $\Sigma^{\textnormal{on}}_t$  as in~\eqref{eq:def Sigonj}, if
  \begin{equation}\label{eq:condition spheres}
    \norm{Z \left(\!\begin{smallmatrix} x(t)\\
          u(t) \end{smallmatrix}\!\right) - x(t+1) } >
    q+r\norm{\!\left(\!\begin{smallmatrix} x(t)\\
          u(t) \end{smallmatrix}\!\right)\!},
  \end{equation}
  then $\Sigma\cap\Sigma^{\textnormal{on}}_t =\emptyset$.
\end{lemma} 
\begin{proof}
  Suppose that \eqref{eq:condition spheres} holds and let
  $\bar{Z}\in\Sigma$. We show that
  $\bar{Z}\not\in\Sigma^{\textnormal{on}}_t$.  From the triangle
  inequality, we have
  \begin{align*}
    &\hspace{-1em} \norm{\bar{Z} \left(\!\begin{smallmatrix} x(t)\\
          u(t) \end{smallmatrix}\!\right) - x(t+1) } \\&\geq \norm{Z
    \left(\!\begin{smallmatrix} x(t)\\ u(t) \end{smallmatrix}\!\right)
    - x(t+1) } - \norm{(\bar{Z}-Z)\left(\!\begin{smallmatrix} x(t)\\
        u(t) \end{smallmatrix}\!\right)\!}.
  \end{align*}
  Since, by assumption $\bar{Z}\in B^r(Z)$, we can write
  \[
    \norm{(\bar{Z}-Z)\left(\!\begin{smallmatrix} x(t)\\
          u(t) \end{smallmatrix}\!\right)\!}\leq
    r\norm{\!\left(\!\begin{smallmatrix} x(t)\\
          u(t) \end{smallmatrix}\!\right)\!}.
  \]
  Combining this with~\eqref{eq:condition spheres}, we obtain
  \[
    \norm{\bar{Z} \left(\!\begin{smallmatrix} x(t)\\
          u(t) \end{smallmatrix}\!\right) - x(t+1) } > q,
  \]
  and hence  $\bar{Z}\not\in\Sigma^{\textnormal{on}}_t$ as claimed.
\end{proof}

This result allows to check incompatibility by means of a scalar
condition, instead of finding the intersections of two quadratic sets
in~$\mathbb{R}^{n\times(n+m)}$.  As a further benefit, we can apply
the previous reasoning to efficiently determining when a
\textit{chosen} input is guaranteed to resolve the incompatibility
problem.

\begin{corollary}[Choosing inputs for incompatibility]\label{cor:input
    design}
  Let $\Sigma^{\textnormal{on}}_t$ be as in \eqref{eq:def Sigonj} and
  suppose that $\Sigma^i\subseteq B^{r_i}(Z_i)$ for all
  $i\in\calP$. Let $i\neq j\in\calP$ and $x(t)\in\mathbb{R}^n$. If
  $u(t)$ is such that
  \begin{equation}\label{eq:input design}
    \norm{ (Z_i-Z_j)\left(\!\begin{smallmatrix} x(t)\\
          u(t) \end{smallmatrix}\!\right)}
    >(r_i+r_j)\norm{\left(\!\begin{smallmatrix} x(t)\\
          u(t) \end{smallmatrix}\!\right)}+2q ,
  \end{equation}
  then at most one of
  $\Sigma^i\cap\Sigma^{\textnormal{on}}_t\neq \emptyset$ or
  $\Sigma^j\cap\Sigma^{\textnormal{on}}_t\neq \emptyset$ holds.
\end{corollary} 

Now, the condition in Corollary~\ref{cor:input design} provides 
conditions under which an input can be picked such that a mode can be 
discarded. If, for each pair of modes and at each state each 
$x\in\mathbb{R}^n$, we were able to do so, we can discard a mode at 
each step of online operation, yielding a procedure which takes at most $p$ steps. 

Interestingly, this can be met wherever the effect of the input on the systems 
in $\Sigma_i$ and $\Sigma_j$ is `different enough'. Formally, if
\begin{equation}\label{eq:sufficient input design}
  \norm{(Z_i-Z_j)\begin{bmatrix} 0\\ I_m \end{bmatrix}}
  >
  r_i+r_j,
\end{equation}
then, for any $x(t)$, there exists a large enough $u(t)$ for
which~\eqref{eq:input design} holds.  Equipped with these results, one can generalize
Algorithm~\ref{alg:mode_detection} to noisy data by first outer approximating the
sets of systems compatible with the initialization data
$\{\Sigma_i\}_{i \in \calP}$ using Lemma~\ref{lem:bounded} and then
(assuming that \eqref{eq:sufficient input design} holds for a sequence of pairs)
designing inputs guaranteed to eliminate at least one mode from
consideration at each step. 
%The corresponding mode detection scheme takes 
%at most $p$ steps. 
Instead of formalizing this procedure, we will 
take full advantage of the information provided by multiple measurements at once.

\paragraph*{Considering incremental measurements}
Treating online measurements separately may prove restrictive. {\color{black}In
fact, the true system is not just contained in each
$\Sigma^{\textnormal{on}}_t$, but in their intersection
$\Sigma^{\textnormal{on}} = \bigcap_{t=0}^{T^{\textnormal{on}}-1}
\Sigma^{\textnormal{on}}_t$.} To address this, we develop here results
dealing with intersections of such sets. In order to circumvent the
computational complexity associated with considering increasing
numbers of intersections of quadratic sets, we overestimate them with
a single set, which leads to conservative, more computationally
efficient tests.  We employ a set representation of the
form~\eqref{eq:Sigma as QMI}, that is, we have for each
$t=0,\ldots,T^{\textnormal{on}}-1$,
\[
  \Sigma^{\textnormal{on}}_t =
  \Bigg \{ (A,B) \mid
  \begin{bmatrix}
    I_n \\
    A^\top
    \\
    B^\top
  \end{bmatrix}^\top
  N^{\textnormal{on}}_t
  \begin{bmatrix}
    I_n \\
    A^\top\\B^\top
  \end{bmatrix}
  \succeq 0 \Bigg \},
\]
for some $N^{\textnormal{on}}_t$ given in the form of \eqref{def:N},
with a corresponding noise model~\eqref{eq:noise model
  restricted}. Given $\alpha_0,\ldots,\alpha_{T^{\textnormal{on}}-1}\geq 0$, we
denote the vector $\alpha := ( \alpha_0\cdots\alpha_{T^{\textnormal{on}}-1})^\top$
and define
$ N^{{\textnormal{on}}}(\alpha) := \sum_{t=0}^{T^{\textnormal{on}}-1}\alpha_t
N^{\textnormal{on}}_t $. Similarly, we let $N^i$ be the matrices
corresponding to the initialization sets $\Sigma^i$, and define
$ N^{i,{\textnormal{on}}}(\alpha) := N^i+
N^{\textnormal{on}}(\alpha)$, for each $i\in\calP$.  Consider the sets
\begin{align*}
  \Sigma^{\textnormal{on}}(\alpha)
  &:= \Bigg \{ (A,B)
    \mid \begin{bmatrix} I_n \\ A^\top
      \\ B^\top \end{bmatrix}^\top
  N^{{\textnormal{on}}}(\alpha)
  \begin{bmatrix}
    I_n
    \\
    A^\top
    \\
    B^\top
  \end{bmatrix}
  \succeq 0 \Bigg \},
  \\
  \Sigma^{i,\textnormal{on}}(\alpha)
  &:=
    \Bigg \{ (A,B)
    \mid \begin{bmatrix}
      I_n \\
      A^\top \\
      B^\top
    \end{bmatrix}^\top
  N^{i,{\textnormal{on}}}(\alpha)
  \begin{bmatrix}
    I_n
    \\
    A^\top\\B^\top
  \end{bmatrix}
  \succeq 0 \Bigg \}.
\end{align*}
Note that, if a number of quadratic matrix inequalities are satisfied,
then so is any nonnegative combination of them. Therefore, for any
$\alpha_0,\ldots,\alpha_{T^{\textnormal{on}}-1}\geq 0$,
\[
  \Sigma^{\textnormal{on}} \subseteq \Sigma^{\textnormal{on}}(\alpha)
  \textrm{ and } \Sigma^i\cap\Sigma^{\textnormal{on}} \subseteq
  \Sigma^{i,\textnormal{on}}(\alpha).
\]
This provides over-approximations of $ \Sigma^{\textnormal{on}}$ and
$ \Sigma^i\cap\Sigma^{\textnormal{on}}$ in terms of such nonnegative
combinations, allowing us to state the following result.

\begin{lemma}[Parameterized condition for incompatibility with online
  measurements]\label{lem:suff cond for incompatibility}
  Given online measurements
  $(U_-^{\textnormal{on}},X^{\textnormal{on}})$ and $i \in \calP$, if
  there exist $\alpha_0,\ldots,\alpha_{T^{\textnormal{on}}-1}\geq 0$
  such that either
  \begin{enumerate}
  \item\label{item:para 1} $\Sigma^i\cap\Sigma^{ \textnormal{on}}(\alpha)= \emptyset$
    or
  \item\label{item:para 2} $\Sigma^{i, \textnormal{on}}(\alpha)= \emptyset$,
  \end{enumerate}
  then $\Sigma^i\cap\Sigma^{\textnormal{on}}=\emptyset$.
\end{lemma}

This result allows us to test intersections of multiple quadratic sets
for emptiness in terms of a parametrized intersection of either two
such sets using Lemma~\ref{lem:suff cond for
  incompatibility}\ref{item:para 1}, or even one,
using Lemma~\ref{lem:suff cond for incompatibility}\ref{item:para 2}. 
The following result provides a way to efficiently test for
the latter.

\begin{lemma}[Spectral test for incompatibility with online
  measurements]
  Given initialization data $\{(U_-^i,X^i)\}_{i \in \calP}$ and online
  measurements $(U_-^{\textnormal{on}},X^{\textnormal{on}})$, suppose
  $\begin{bmatrix} X_-^i\\U_-^i\end{bmatrix}$ has full row rank for
  $i\in\calP$.  For any $i\in \calP$ and
  $\alpha_0,\ldots,\alpha_{T^{\textnormal{on}}-1}\geq 0$,
  $\Sigma^{i, \textnormal{on}}(\alpha)\neq \emptyset$ if and only if
  $N^{i,\textnormal{on}}(\alpha)$ has precisely $n+m$ negative
  eigenvalues. Equivalently, given
  \[
    N^{i,{\textnormal{on}}}(\alpha) =
    \begin{bmatrix}
      \bar{N}_{11} & \bar{N}_{12} \\ \bar{N}_{21} & \bar{N}_{22}
    \end{bmatrix},
  \]
  then $\Sigma^{i, \textnormal{on}}(\alpha)\neq \emptyset$ if and only
  if $\bar{N}_{11} -\bar{N}_{12}\bar{N}_{22}^{-1}\bar{N}_{21}\succeq 0$.
\end{lemma}
\begin{proof}
  Given that $\bar{N}_{22}\prec 0$ by construction, the statement follows
  from~\cite[Thm. 3.2]{HJVW-MKC-JE-HLT:22}.
\end{proof}

One can use either criteria in Lemma~\ref{lem:suff cond for
  incompatibility} to generalize Algorithm~\ref{alg:mode_detection} to
noisy data in a way that integrates the information provided by
multiple measurements at once. {\color{black}For instance, in line with
  Corollary~\ref{cor:ens ass 3}, we could formulate a conservative
  test for condition \ref{item:para 1} of Lemma~\ref{lem:suff cond for
    incompatibility}, using only radii and centers. However, in the
  following, we derive a less conservative test for the same
  condition. }

Let $\alpha = \mathbbm{1} \in \mathbb{R}^{T^{\textnormal{on}}}$,
the vector of all ones. Recall that, during online operation, we
collect single measurements, with each noise sample satisfying
$\norm{w(t)} \leq q$, for some $q>0$. The set
$\Sigma^{\textnormal{on}}_t$ is then described by~\eqref{eq:def
  Sigonj}.  Consider
  \[
    N^{\textnormal{on}}(\mathbbm{1}) = \begin{bmatrix} I_n\!\!
      &X^{\textnormal{on}}_+\\0\!\!&-X^{\textnormal{on}}_-\\0\!\!&-U^{\textnormal{on}}_-
  \end{bmatrix}
  \!\!
  \begin{bmatrix}
    q^2T^{\textnormal{on}}I_n & 0 \\ 0&
    -I_{T^{\textnormal{on}}}
  \end{bmatrix}
  \!\!
  \begin{bmatrix}
    I_n \!\!&X^{\textnormal{on}}_+\\0\!\!&-X^{\textnormal{on}}_-\\0\!\!&-U^{\textnormal{on}}_-
  \end{bmatrix}^\top .
\]
Using the Schur complement, we can conclude that
$(A,B) \in\Sigma^{\textnormal{on}}(\mathbbm{1})$ if and only if
\[
  \begin{bmatrix} q^2 T^{\textnormal{on}} I &
    X^{\textnormal{on}}_+-AX^{\textnormal{on}}_--BU^{\textnormal{on}}_-
    \\
    (X^{\textnormal{on}}_+-AX^{\textnormal{on}}_--BU^{\textnormal{on}}_-)^\top&
    I
  \end{bmatrix}\succeq 0.
\]
A similar LMI can be obtained to check whether
$(A,B) \in\Sigma^i$, leading to the following result.

\begin{corollary}[LMI test for incompatibility with online
  measurements]\label{cor:noisy stuff}
  Given online single measurements obtained sequentially, with each
  noise sample satisfying $\norm{w(t)} \leq q$, for some $q>0$, let
  $i \in \calP$ and assume the initialization data $(U_-^i,X^i)$ has
  noise model~\eqref{eq:noise model restricted} with $Q^i = q^2T^i I_n$.
  If there are no matrices $A$ and $B$ that satisfy simultaneously
  \begin{subequations}\label{noisy_LMI}
    \begin{align}
        \begin{bmatrix}
        q^2T^iI_n&X^i_+-AX^i_--BU^i_-
        \\
        (X^i_+-AX^i_--BU^i_-)^\top& I_{T^i}
      \end{bmatrix}\! & \succeq \! 0,
      \\
      \begin{bmatrix}
        q^2T^{\textnormal{on}}I_n&X^{\textnormal{on}}_+-AX^{\textnormal{on}}_--B
        U^{\textnormal{on}}_-
        \\
        (X^{\textnormal{on}}_+-AX^{\textnormal{on}}_--BU^{\textnormal{on}}_-)^\top&
        I_{T^{\textnormal{on}}}
      \end{bmatrix} \! & \succeq \! 0, 
    \end{align}
  \end{subequations}
  then the data are incompatible, that is,
  $\Sigma^i\cap\Sigma^{\textnormal{on}}=\emptyset$.
\end{corollary}

\begin{algorithm}[htb!]
  \caption{Mode detection for noisy data}\label{alg:mode_detection_2}
  % \caption{Each time instance of the mode detection algorithm for noisy data}\label{alg:mode_detection_2}
  \algorithmicrequire
  $\calP_{\textrm{match}},\{U^i_-,X^i\}_{i\in\calP},U^{\textnormal{on}}_-,
  X^{\textnormal{on}},c,q$
  \\
  \algorithmicensure
  $\calP_{\textrm{match}},U^{\textnormal{on}}_-,X^{\textnormal{on}}
%  ,f_{\textup{done}}
$
  \begin{algorithmic}[1]           
    % \If{$T^\textnormal{on}\geq t_{\max}$}\LineComment{Maximal online data length}
    % \State $f_{\textup{done}}=1$
    % \Else
    \State Pick a random $u(t)\in\R^m$ such that $\norm{u(t)}\leq c\norm{x(t)}$
    \State Get the next state $x(t+1)$
    \State
    $U_-^{\textnormal{on}}\gets \begin{bmatrix}U_-^{\textnormal{on}}&u(t)\end{bmatrix}$                          
    \State
    $X^{\textnormal{on}}\gets \begin{bmatrix}X^{\textnormal{on}}&
      x(t+1)\end{bmatrix}$ \Comment{Append the data} 
    \For{$i\in\calP_{\textrm{match}}$}  
    \If{the pair of LMIs \eqref{noisy_LMI} is infeasible }                         \LineComment{Data are incompatible with mode $i$}
    \State $\calP_{\textrm{match}}=\calP_{\textrm{match}}\backslash\{i\}$ \Comment{Eliminate mode $i$}
    \EndIf
    \EndFor 
%    \If{$|\calP_{\textrm{match}}|=1$}\LineComment{Mode is detected if only one matches}
%    \State $f_{\textup{done}}=1$
%    \Else
%    \State $f_{\textup{done}}=0$
%    \EndIf
    % \EndIf
  \end{algorithmic}
\end{algorithm}

Algorithm~\ref{alg:mode_detection_2} presents a mode detection
procedure that employs these LMI conditions (instead of the kernel
condition \eqref{eq:kernel} in Algorithm~\ref{alg:mode_detection}) to
check for incompatibility. Also, the inputs at each step of the
algorithm are chosen randomly.

\section{Stabilization phase and detecting
  switches}\label{sec:stabilization}

Having solved the mode detection problem, here we turn our attention
to the stabilization phase.  At the start of this phase, we know the
mode the system is operating in. Therefore, we simply apply the
controller computed in the initialization step, see
Section~\ref{sec:initialization},
\begin{align*}
  u=K_ix ,  
\end{align*}
where $i \in \calP$ is the currently active mode.  The controller
remains in the stabilization phase until a switch in the mode of the
system is detected.  As mentioned in the problem formulation,
in Section~\ref{sec:notion}, we consider two scenarios:
\begin{itemize}
\item the switching signal is partially known: the controller is aware
  of when the system switches modes, yet it does not know which mode
  the system has switched into;
\item the switching signal is completely unknown: this means that the
  controller needs to implement a mechanism to detect whether a switch
  has occurred.
\end{itemize}
In either case, once a switch has been detected, to determine the
current active mode, the controller will switch back to the mode
detection phase described in Section~\ref{sec:MD}.

Here we describe a procedure to detect switches in the system mode in
the second scenario above. This consists of monitoring the evolution
of the Lyapunov function, $V_i(x)=\sqrt{x^\top P_i x}$, computed in the
initialization step, cf. Section~\ref{sec:initialization}, associated
with the currently active mode $i \in \calP$.  From the triangle
inequality and~\eqref{eqn:informativity for uniform stabilization},
when the system is operating in mode $i$ with the feedback control
$u=K_ix$, and under the noise bound $\norm{w(t)}\leq q$, we have
\begin{align}
  V_i&(x(t+1))=\norm{P_i^{\frac{1}{2}}((\hat{A}_i+\hat{B}_iK_i)x(t)+w(t))}
       \nonumber
  \\
  %	&\leq \norm{P_i^{\frac{1}{2}}(A_i+B_iK_i)x(t)}+\norm{P_i^{\frac{1}{2}}w(t)}\\
  %	&=\sqrt{x(t)^\top(A_i+B_iK_i)^\top P_i(A_i+B_iK_i)x(t)}+\norm{P_i^{\frac{1}{2}}w(t)}\\
  %	&\leq \sqrt{\lambda x(t)^\top P_ix(t)}+\norm{P_i^{\frac{1}{2}}}\norm{w(t)}\\
  %	&\leq \sqrt{\lambda}\norm{P_i^{\frac{1}{2}}x(t)}+\lambda_{\max}(P_i^{\frac{1}{2}})q\\
     &\leq{\lambda}V_i(x(t))+\lambda_{\max}(P_i^{\frac{1}{2}})q .
       \label{what_to_obey}
\end{align}
The contrapositive of this statement reads as follows: if the
  inequality does \textit{not} hold, then the system cannot be
  operating in mode~$i$.  We employ it as our switch-detection
  mechanism; that is, the mode detection phase is triggered whenever
\begin{equation}\label{eq:switch detection}
  V_i(x(t+1))> {\lambda}V_i(x(t))+\lambda_{\max}(P_i^{\frac{1}{2}})q.
\end{equation}
Note that the system~\eqref{def:sys} may switch between modes while
the inequality \eqref{what_to_obey} is still preserved (i.e., the
  controller designed for the previously detected mode is still
  stabilizing for the current one). In this case, the controller
applies the controller corresponding to the previously detected mode,
and the mode of the system and the mode of the controller become
desynchronized. As our ensuing technical analysis shows, this does not
affect the stability properties of the closed-loop system.
% given the fact that \eqref{what_to_obey} holds.
% This means that the Lyapunov function decreases at the rate $\lambda$
% regardless of this desynchronization.

\section{Stabilization via online switched controller}\label{sec:osc}
In this section, we describe our online switched controller design and
establish the asymptotic convergence properties of the resulting
closed-loop switched system.

\subsection{Switched controller design}\label{sec:controller}
Our controller combines the mode detection and the stabilization
phases into a single design, formalized in
Algorithm~\ref{alg:controller_1}, for the case of noisy data and an
unknown switching signal. Next, we provide an intuitive description of
the pseudocode language:
\begin{quote}
  \emph{Informal description:} The algorithm assumes an initialization
  step satisfying Assumptions~\ref{ass:ass_on_data-I}
  and~\ref{ass:ass_on_data-II} has been performed. The algorithm
  starts in the mode detection phase, indicated by the variable
  $S_{\textrm{phase}}$. During this phase,
  Algorithm~\ref{alg:mode_detection_2} is executed for each iteration
  until mode detection is successful. The variable $\sigma_d$ is then taken to be
  the active mode of the system. Next, the algorithm switches to the
  stabilization phase (signaled by $S_{\textrm{phase}}=1$).  During
  this phase, the control input $u(t)=K_{\sigma_d}x(t)$ is applied
  until \eqref{eq:switch detection} holds. This marks the detection of
  a mode switch, which leads the controller to go back to the mode
  detection phase (by toggling $S_{\textrm{phase}}$ to $0$). In the
  meantime, $U_-^{\textnormal{on}}$ and $X^{\textnormal{on}}$ are reset to
  record the new online data.
\end{quote}

\begin{algorithm}[htb!]
  \caption{Data-driven online switched feedback controller for noisy
    data with unknown switching signal}\label{alg:controller_1}
  \algorithmicrequire $\calP,\{U^i_-,X^i,K_i,P_i\}_{i\in\calP},\lambda,c,q$
  %\algorithmicensure $$
  \begin{algorithmic}[1]
%  \State $p_{\max,i}\gets\lambda_{\max}(P_i^{\frac{1}{2}})$
    \State $\calP_{\textrm{match}}\gets\calP$
  \State $S_{\textrm{phase}}$ $\gets 0$ \Comment{Initialize to mode detection phase}
  % \State Pick initial mode $\sigma_d\in\calP$
  %
  \State $X^{\textnormal{on}}\gets x(0)$
  \State $U_-^{\textnormal{on}}\gets []$
  \Comment{Initialize online data}
  \While{the system \eqref{def:sys} is running}
  \If{$S_{\textrm{phase}}$ $=0$}                    \Comment{Mode
    detection phase}
  \State\label{step: algorithm} Run Algorithm~\ref{alg:mode_detection_2} to update $U^{\textnormal{on}}_-$, $X^{\textnormal{on}}$, and $\calP_{\textrm{match}}$
  % $f_{\textup{done}}$
  % \If{$f_{\textup{done}}=1$}
  \If{$|\calP_{\textrm{match}}|=1$}
  \State Pick
  $\sigma_d\in\calP_{\textrm{match}}$ \label{step:mode_to_switch}
  \Comment{Set the controller mode}
  \State $S_{\textrm{phase}}$ $\gets 1$                \LineComment{Change controller to stabilization phase}
  % \Else   \If{$\calP_{\textrm{match}}=\emptyset$}
%  \State Terminate with error      \Comment{The case when the online data is incompatible with all initialization data}
%  \EndIf
  \EndIf    
  \Else   \Comment{Stabilization phase}
  \State Apply control $u(t)=K_{\sigma_d}x(t)$ %\Comment{Apply the control with mode $\sigma$}
  \State Obtain the next state $x(t+1)$
  \If{\eqref{eq:switch detection} holds} \label{step:switch detec}   \Comment{Trigger mode-detection}
 \State $U_-^{\textnormal{on}}\gets u(t)$ \Comment{Reset the online data}
 \State $X^{\textnormal{on}}\gets \begin{bmatrix} x(t)&x(t+1)\end{bmatrix}$                 
 \State $\calP_{\textrm{match}}\gets\calP$	\Comment{Reset $\calP_{\textnormal{match}}$}
 \State $S_{\textrm{phase}}$ $\gets 0$             \LineComment{Change controller to mode detection phase}
 \EndIf
 \EndIf
    \State $t\gets t+1$                     \Comment{Update the time}
    \EndWhile
  \end{algorithmic}
\end{algorithm}

Note that the closed-loop system is stable when the controller is in
its stabilization phase. However, during the mode detection phase, the
effect of the controller on the system might be destabilizing because
of the potential mismatch with the system mode.  Therefore,
determining overall stability properties of the
form~\eqref{ISS-estimate_0} relies critically on the switching
behavior, both of the system and the controller.

\begin{remark}[Data-driven online switched feedback controller for
  noiseless data and known switching signal]
  Algorithm~\ref{alg:controller_1} requires minor modifications in (i)
  the noiseless case or (ii) when the switching signal is known. For
  (i), one replaces Algorithm~\ref{alg:mode_detection_2} in
  Step~\ref{step: algorithm} with Algorithm~\ref{alg:mode_detection}
  and takes $q=0$ when evaluating~\eqref{eq:switch detection} in
  Step~\ref{step:switch detec}.  For (ii), one replaces the check in
  Step~\ref{step:switch detec} by the condition
  $\sigma(t+1)\neq\sigma(t)$. \oprocend
\end{remark}

\begin{remark}[Appending the online data to initialization
  measurements]\label{rem:appending}
  Once the mode detection phase is successful, we know that the
  gathered online measurements were generated by the active mode
  $i\in\calP_{\textnormal{match}}$. This raises the possibility of
  incorporating such online data to the initialization measurements
  corresponding to the mode $i\in\calP_{\textnormal{match}}$.
  % that is, we add a command
  % $\Sigma^i \leftarrow
  % \Sigma^i\cap\Sigma^{\textnormal{on}}$.
  Intuitively, this would decrease the size of the set $\Sigma^i$,
  which means that future mode detection phases would require fewer
  online measurements at the cost of having the controller keep
  previous measurements in its memory. \oprocend
\end{remark}

\subsection{{\color{black} Assumptions on dwell- and activation-times
    of switching}}\label{sec:dwell}

In this section, we introduce some assumptions on the
switching. Intuitively, a switched system can be stabilized if
switching is infrequent and the mode-detection phase, which is usually
unstable, is relatively short. Our conditions for the switching signal
are characterized through \emph{average dwell-time} (ADT)
condition~\cite{JPH-AM:99} on the switching and \emph{average
  activation-time} (AAT) condition~\cite{MM-DL:12} on the mode
detection phase, which are summarized as follows.
% Note that the behavior of the closed-loop system depends on both the
% switching signal $\sigma$ of the system \eqref{def:sys}, but also on
% the switches of the controller, as given by
% Algorithm~\ref{alg:controller_1}.

\begin{assumption}[Properties of the switching signal of the
  controller]\label{ass:detection_after_exercution}
  Let $\bT^m:=\{t_1^m,t_2^m,\cdots\}$ be the ordered set consisting of
  the initial time instants of each mode detection phase. Similarly,
  let $\bT^s:=\{t_1^s,t_2^s,\cdots\}$ correspond to the time instants
  when the stabilization phase starts.  Assume
  \begin{enumerate}
  \item\label{itm:asump 1} the system does not switch while the
    controller is in the mode detection phase;
    % that is, 
    % \[t_1^m<t_1^s<t_2^m<t_2^s<\cdots.\]
  \item\label{itm:asump adt} let $\calN(t_a,t_b)$ be the total
    number of mode detection phases over the time interval
    $[t_a,t_b)$, that is, $ \calN(t_a,t_b):= |[t_a,t_b)\cap
    \bT^m|$. There exists $\tau>0$ and $N_0\geq 1$ such that
    \begin{equation}\label{ADT_condition}
      \calN(t_a,t_b) \leq N_0+\frac{t_b-t_a}{\tau} ,\quad\forall t_a,t_b\in
      \N, t_a< t_b;
    \end{equation}
    
  \item\label{itm:asump aat} let $\calM(t_a,t_b)$ be the total
    time spent in mode detection phases over the time interval
    $[t_a,t_b)$, that is,
    $\calM(t_a,t_b):=\sum_{t=t_a}^{t_b-1}{\mathbf 1}(t)$, where
    \[
      \mathbf 1 (t):=\begin{cases} 1& \mbox{ if }t\in [t_i^m,t_i^s)
        \mbox{ for some }i\in\N,
        \\
	0&\mbox{ otherwise. }
      \end{cases}
    \] 
    Then, there exists $\eta\in[0,1]$ and $T_0\geq 0$ such that
    \begin{equation}\label{AAT_condition}
      \calM(t_a,t_b) \leq T_0+\eta(t_b-t_a)\quad\forall
      t_a,t_b\in \N, t_a< t_b. 
    \end{equation}
  \end{enumerate}
\end{assumption}

%We discuss next
%Assumption~\ref{ass:detection_after_exercution}. 
Recall that by construction, the controller alternates between the
mode detection phase and the stabilization phase. In other words, the
elements in $\bT^m$ and $\bT^s$ are ordered such that
\[
  t_1^m<t_1^s<t_2^m<t_2^s<\ldots
\]

Statement~\ref{itm:asump 1} ensures that the collected online
measurements correspond to a single active mode.
%From
%Corollary~\ref{cor:at-most-n+m}, the time length of each mode
%detection phase for noiseless data is bounded above by $n+m$. In
%simulations, we have observed that regardless of whether the data are
%noisy or noiseless, the mode detection phase is much shorter than this
%upper bound.  This means that, as long as the unknown switching signal
%does not switch too frequently, statement~\ref{itm:asump 1} holds in
%general.
Statement~\ref{itm:asump adt} ensures that, on average, the controller
is switched to the mode detection phase no more than once per $\tau$
time instances.
%Clearly, this condition holds if $\tau$ is relatively
%large, or equivalently, if the controller switches infrequently.
Similarly, statement~\ref{itm:asump aat} ensures that, on average, the
controller dwells in the mode detection phase for at most a fraction
$\eta$ of the total time. Intuitively,
Assumption~\ref{ass:detection_after_exercution} holds if the controller
switches infrequently and if the mode detection phase is, on average,
relatively short when compared to the stabilization phase.

Note that Assumption~\ref{ass:detection_after_exercution} is
formulated in terms of the switching signal $\sigma_d$ associated to
the controller. In practice, this can be difficult to verify. However,
one can guarantee that these conditions hold by making slightly
stronger assumptions on the system switching signal $\sigma$, the
switch detection, and the mode detection, as follows.

{\color{black}
  \begin{lemma}[Properties of the system switching signal and the
    controller]\label{ass3'} 
    Let
    $\bT:=\{t_1,t_2,\cdots\}=\{t\in\N\backslash\{0\}:\sigma(t)\neq\sigma(t-1)\}$
    denote the ordered set of time instances at which a system switch
    occurs, and let $\bT^m, \bT^s$ be as defined in
    Assumption~\ref{ass:detection_after_exercution}. Assume there
    exist $\tau,\tau_a,\tau_b\in\N$ with $\tau>\tau_a+\tau_b$ such
    that
    \begin{subequations}\label{eq:ass3'}
      \begin{align}
        t_{i+1}-t_i&\geq \tau,\label{ass3'_1}
        \\
        t^m_i-t_i&\leq \tau_a,\label{ass3'_2}
        \\
        t^s_i-t^m_i&\leq \tau_b, \label{ass3'_3}
      \end{align}
    \end{subequations}
    for all $i\in\N\backslash\{0\}$. Then
    Assumption~\ref{ass:detection_after_exercution} holds.
  \end{lemma}
  \begin{proof}
    The proof follows routinely, by noting that condition
    \eqref{ass3'_1} implies \eqref{ADT_condition} with the same
    parameter $\tau$ and $N_0=1$; conditions \eqref{ass3'_2} and
    \eqref{ass3'_3} imply \eqref{AAT_condition} with
    $\eta=\frac{\tau_b}{\tau}$ and $T_0=\tau_b$; and, with
    $\tau>\tau_a+\tau_b$, we have that
    Assumption~\ref{ass:detection_after_exercution}\ref{itm:asump 1}
    holds.
  \end{proof}
  
  Regarding Lemma~\ref{ass3'}, note that the first condition
  \eqref{ass3'_1} means that the switching signal $\sigma$ has a
  minimal dwell time of~$\tau$, and it is a way of formalizing the
  idea that the system switches sufficiently infrequently. The second
  condition \eqref{ass3'_2} means the mode switch is detected no later
  than $\tau_a$ steps after $\sigma$ changes its value. For scenarios
  where the switching instants are known, this condition holds
  immediately.
% Meanwhile, for noiseless systems, this condition is true in general,
% as incompatible control usually immediately leads to divergence of
% the state and violation of the condition~\eqref{eq:switch detection}
% (i.e., $\tau_a=1$).
The last condition \eqref{ass3'_3} means that each mode-detection
phase will last no more than $\tau_b$ time steps. According to
Corollary~\ref{cor:at-most-n+m}, this can be guaranteed with
$\tau_b=n+m$ for noiseless measurements.
% In simulations, we have observed that regardless of whether the data
% are noisy or noiseless, \eqref{ass3'_2} and \eqref{ass3'_3} always
% hold with small values of $\tau_a$ and $\tau_b$.

}

\subsection{Stability analysis of the closed-loop
  system}\label{sec:stability}

The following result characterizes the stability properties of the
closed-loop system. 

\begin{theorem}[Stability guarantee for the closed-loop
  system]\label{thm:stability_control}
  Consider initialization data $(U_-^i,X^i)$ with, for each
  $i\in\calP$, noise model of the form~\eqref{eq:noise model
    restricted}, $Q^i = q^2T^i I_n$, for some $q>0$, and satisfying
  Assumptions~\ref{ass:ass_on_data-I} and~\ref{ass:ass_on_data-II}.
  Let online measurements be collected sequentially, with
  $\norm{w(t)}\leq q$ at each time $t$.  Assume that the switching
  signal of the controller satisfies
  Assumption~\ref{ass:detection_after_exercution}.  Let
  \begin{equation}\label{def:mu}
    \mu:=\max_{i,j\in\calP}\Vert P_i^{\frac{1}{2}}P_j^{-\frac{1}{2}}\Vert,
  \end{equation}
  and $\lambda_u\geq 1$ be such that for any $i\in\calP$, there exists
  $k_i\geq0$ such that
  \begin{equation}\label{LMI_for_MD}
    \begin{bmatrix}
      \lambda_u^2P_i-k_i c^2 I_n&0&\hat{A}_i^\top\\
      0&k_iI_m&\hat{B}_i^\top\\
      \hat{A}_i&\hat{B}_i&P_i^{-1}
    \end{bmatrix}\succeq 0.
  \end{equation}

  If the following holds
  \begin{equation}\label{ADT_AAT_condition}
    \left(1-\frac{\ln\lambda_u}{\ln\lambda}\right)\eta +
    \left(1-\frac{\ln\mu}{\ln\lambda}\right)\frac{1}{\tau}<1,  
  \end{equation}
  then, for all initial states $x(0)\in\R^n$ and each time $t\in\N$,
  the solution of the closed-loop system~\eqref{def:sys} with the
  data-driven switching feedback controller described by
  Algorithm~\ref{alg:controller_1} satisfies
  \begin{equation}\label{ISS-estimate}
    \norm{x(t)} \leq \frac{p_{\max}}{p_{\min}}a^tb\norm{x(0)} +
    \frac{p_{\max}}{p_{\min}}\frac{ab}{(1-a){\lambda}}q  ,
  \end{equation}
  where
  \begin{subequations}\label{eq:a-b}
    \begin{align}
      a&:={\lambda\left(\frac{\mu}{\lambda}\right)^{\frac{1}{\tau}}
         \left(\frac{\lambda_u}{\lambda}\right)^{\eta}}
         \in(\lambda,1),
         \label{def:a} 
      \\ 
      b&:={\left(\frac{\mu}{\lambda}\right)^{N_0}
         \left(\frac{\lambda_u}{\lambda}\right)^{T_0}} , \label{def:b} 
    \end{align}
  \end{subequations}
  and $p_{\max}=\max_{i\in\calP}\lambda_{\max}(P_i^{\frac{1}{2}})$,
  $ p_{\min}=\min_{i\in\calP}\lambda_{\min}(P_i^{\frac{1}{2}})$.
\end{theorem}

The proof of Theorem~\ref{thm:stability_control} is provided in the
Appendix. The proof uses similar techniques as other methods which
  employ multiple Lyapunov functions. Such methods are common for
  stability analysis of switched systems~\cite{MM-DL:12,
    SL-AD-DL:21}.

% Intuitively, the closed-loop system has the ISS-like property if the
% switching occurs sufficiently infrequently and the time of unstable
% mode detection phase is relatively short.}

We remark that condition~\eqref{LMI_for_MD} refers to the original
  true system and checking it requires knowledge of
  $\hat A_i, \hat B_i$, which are not readily available in our
  data-driven setup. However, this is not critical for the analysis,
  since Theorem~\ref{thm:stability_control} can be interpreted
  qualitatively: even if $\lambda_u$ and $k_i$ are not known a priori, 
  this LMI holds for each $i\in\calP$, for sufficiently large 
  $\lambda_u$ and $k_i$.

Given Assumption~\ref{ass:ass_on_data-I}, $\lambda < 1$ is an upper
bound on the decay rate during the stabilization phase. Instead,
$\lambda_u \ge 1$ is an upper bound on the growth rate during the mode
detection phase. The parameter $\mu$ corresponds to the destabilizing
effect introduced by each switching. If the subsystems are each very different from each other, then this could lead to significantly different $P_i$ matrices, leading to large values of $\mu$. The interpretation of
condition~\eqref{ADT_AAT_condition} is as follows: the first term
quantifies the combined effect on the growth rate of the state of the
mode detection and stabilization phases. The second term quantifies
the destabilizing effect of mode switches, which happen on average
every $\tau$ time instances.  For a given switched system, the
constants $\mu$, $\lambda$, and $\lambda_u$ are fixed. Hence, the
condition~\eqref{ADT_AAT_condition} is always satisfied if $\tau$ is
sufficiently large and $\eta$ is sufficiently small. This means that
practical exponential stability~\eqref{ISS-estimate} holds as
long as the system switches sufficiently infrequently and the mode
detection phases are sufficiently short.

Lastly, recall that the parameter $c$ constitutes a bound on the relative magnitude of the excitation inputs with respect to the state signal. A good choice of a value for $c$ is thus a balancing act between two objectives. Choosing a larger $c$ increases the destabilizing effect of the mode-detection phase. On the other hand, a small $c$ can also cause problems. For the noiseless case, numerical problems (e.g., in the verification of the kernel inclusion (10)) may occur when $c$ is too small. For the noisy case, the applied control may be dominated by the disturbance $w$ when $c$ is too small, resulting in insufficiently excited data.

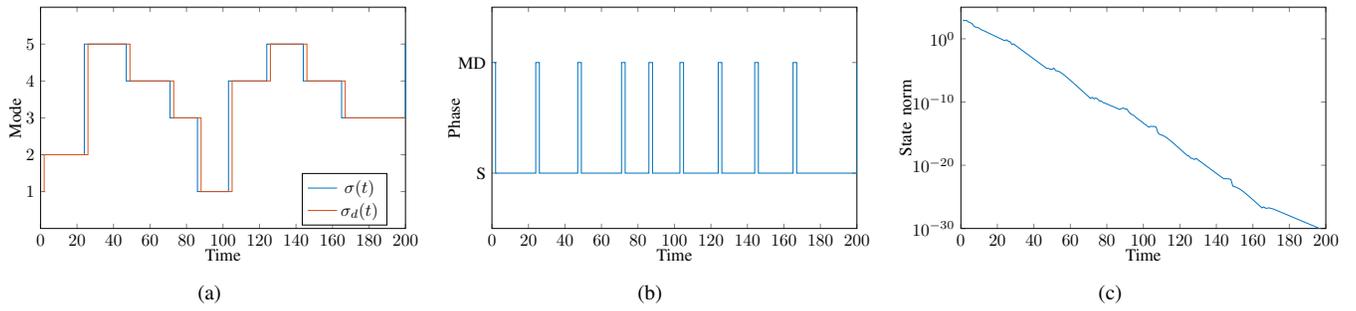
\begin{figure*}[ht!]
  \tikzset{every picture/.style={scale=0.65}}
  \centering
  \subfigure[]{% This file was created by matlab2tikz.
\definecolor{mycolor1}{rgb}{0.00000,0.44700,0.74100}%
\definecolor{mycolor2}{rgb}{0.85000,0.32500,0.09800}%
\begin{tikzpicture}

\begin{axis}[%
width=4.521in,
height=3.566in/1.3,
at={(0.758in,0.481in)},
scale only axis,
xmin=0,
xmax=200,
ymin=0,
ymax=6,
ytick ={1,2,3,4,5},
xlabel={Time},
ylabel={Mode},
xlabel near ticks,
ylabel near ticks,
axis background/.style={fill=white},
legend style={at={(axis cs:190,1.5)},anchor=north east}
]

\addplot[const plot, color=mycolor1] table[row sep=crcr] {%
1	2\\
2	2\\
3	2\\
4	2\\
5	2\\
6	2\\
7	2\\
8	2\\
9	2\\
10	2\\
11	2\\
12	2\\
13	2\\
14	2\\
15	2\\
16	2\\
17	2\\
18	2\\
19	2\\
20	2\\
21	2\\
22	2\\
23	2\\
24	5\\
25	5\\
26	5\\
27	5\\
28	5\\
29	5\\
30	5\\
31	5\\
32	5\\
33	5\\
34	5\\
35	5\\
36	5\\
37	5\\
38	5\\
39	5\\
40	5\\
41	5\\
42	5\\
43	5\\
44	5\\
45	5\\
46	5\\
47	4\\
48	4\\
49	4\\
50	4\\
51	4\\
52	4\\
53	4\\
54	4\\
55	4\\
56	4\\
57	4\\
58	4\\
59	4\\
60	4\\
61	4\\
62	4\\
63	4\\
64	4\\
65	4\\
66	4\\
67	4\\
68	4\\
69	4\\
70	4\\
71	3\\
72	3\\
73	3\\
74	3\\
75	3\\
76	3\\
77	3\\
78	3\\
79	3\\
80	3\\
81	3\\
82	3\\
83	3\\
84	3\\
85	3\\
86	1\\
87	1\\
88	1\\
89	1\\
90	1\\
91	1\\
92	1\\
93	1\\
94	1\\
95	1\\
96	1\\
97	1\\
98	1\\
99	1\\
100	1\\
101	1\\
102	1\\
103	4\\
104	4\\
105	4\\
106	4\\
107	4\\
108	4\\
109	4\\
110	4\\
111	4\\
112	4\\
113	4\\
114	4\\
115	4\\
116	4\\
117	4\\
118	4\\
119	4\\
120	4\\
121	4\\
122	4\\
123	4\\
124	5\\
125	5\\
126	5\\
127	5\\
128	5\\
129	5\\
130	5\\
131	5\\
132	5\\
133	5\\
134	5\\
135	5\\
136	5\\
137	5\\
138	5\\
139	5\\
140	5\\
141	5\\
142	5\\
143	5\\
144	4\\
145	4\\
146	4\\
147	4\\
148	4\\
149	4\\
150	4\\
151	4\\
152	4\\
153	4\\
154	4\\
155	4\\
156	4\\
157	4\\
158	4\\
159	4\\
160	4\\
161	4\\
162	4\\
163	4\\
164	4\\
165	3\\
166	3\\
167	3\\
168	3\\
169	3\\
170	3\\
171	3\\
172	3\\
173	3\\
174	3\\
175	3\\
176	3\\
177	3\\
178	3\\
179	3\\
180	3\\
181	3\\
182	3\\
183	3\\
184	3\\
185	3\\
186	3\\
187	3\\
188	3\\
189	3\\
190	3\\
191	3\\
192	3\\
193	3\\
194	3\\
195	3\\
196	3\\
197	3\\
198	3\\
199	3\\
200	5\\
201	5\\
};
\addlegendentry{$\sigma(t)$}

\addplot[const plot, color=mycolor2] table[row sep=crcr] {%
1	1\\
2	2\\
3	2\\
4	2\\
5	2\\
6	2\\
7	2\\
8	2\\
9	2\\
10	2\\
11	2\\
12	2\\
13	2\\
14	2\\
15	2\\
16	2\\
17	2\\
18	2\\
19	2\\
20	2\\
21	2\\
22	2\\
23	2\\
24	2\\
25	2\\
26	5\\
27	5\\
28	5\\
29	5\\
30	5\\
31	5\\
32	5\\
33	5\\
34	5\\
35	5\\
36	5\\
37	5\\
38	5\\
39	5\\
40	5\\
41	5\\
42	5\\
43	5\\
44	5\\
45	5\\
46	5\\
47	5\\
48	5\\
49	4\\
50	4\\
51	4\\
52	4\\
53	4\\
54	4\\
55	4\\
56	4\\
57	4\\
58	4\\
59	4\\
60	4\\
61	4\\
62	4\\
63	4\\
64	4\\
65	4\\
66	4\\
67	4\\
68	4\\
69	4\\
70	4\\
71	4\\
72	4\\
73	3\\
74	3\\
75	3\\
76	3\\
77	3\\
78	3\\
79	3\\
80	3\\
81	3\\
82	3\\
83	3\\
84	3\\
85	3\\
86	3\\
87	3\\
88	1\\
89	1\\
90	1\\
91	1\\
92	1\\
93	1\\
94	1\\
95	1\\
96	1\\
97	1\\
98	1\\
99	1\\
100	1\\
101	1\\
102	1\\
103	1\\
104	1\\
105	4\\
106	4\\
107	4\\
108	4\\
109	4\\
110	4\\
111	4\\
112	4\\
113	4\\
114	4\\
115	4\\
116	4\\
117	4\\
118	4\\
119	4\\
120	4\\
121	4\\
122	4\\
123	4\\
124	4\\
125	4\\
126	5\\
127	5\\
128	5\\
129	5\\
130	5\\
131	5\\
132	5\\
133	5\\
134	5\\
135	5\\
136	5\\
137	5\\
138	5\\
139	5\\
140	5\\
141	5\\
142	5\\
143	5\\
144	5\\
145	5\\
146	4\\
147	4\\
148	4\\
149	4\\
150	4\\
151	4\\
152	4\\
153	4\\
154	4\\
155	4\\
156	4\\
157	4\\
158	4\\
159	4\\
160	4\\
161	4\\
162	4\\
163	4\\
164	4\\
165	4\\
166	4\\
167	3\\
168	3\\
169	3\\
170	3\\
171	3\\
172	3\\
173	3\\
174	3\\
175	3\\
176	3\\
177	3\\
178	3\\
179	3\\
180	3\\
181	3\\
182	3\\
183	3\\
184	3\\
185	3\\
186	3\\
187	3\\
188	3\\
189	3\\
190	3\\
191	3\\
192	3\\
193	3\\
194	3\\
195	3\\
196	3\\
197	3\\
198	3\\
199	3\\
200	3\\
};
\addlegendentry{$\sigma_d(t)$}

\end{axis}
\end{tikzpicture}%
}
  \subfigure[]{
\definecolor{mycolor1}{rgb}{0.00000,0.44700,0.74100}%
\begin{tikzpicture}

\begin{axis}[%
width=4.521in,
height=3.566in/1.3,
at={(0.758in,0.481in)},
scale only axis,
xmin=0,
xmax=200,
ymin=-0.5,
ymax=1.5,
ytick={0,1},
yticklabels={S,MD},
xlabel={Time},
ylabel={Phase},
xlabel near ticks,
ylabel near ticks,
axis background/.style={fill=white},
legend style={legend cell align=left, align=left, draw=white!15!black}
]
\addplot[const plot, color=mycolor1] table[row sep=crcr] {%
1	1\\
2	0\\
3	0\\
4	0\\
5	0\\
6	0\\
7	0\\
8	0\\
9	0\\
10	0\\
11	0\\
12	0\\
13	0\\
14	0\\
15	0\\
16	0\\
17	0\\
18	0\\
19	0\\
20	0\\
21	0\\
22	0\\
23	0\\
24	1\\
25	1\\
26	0\\
27	0\\
28	0\\
29	0\\
30	0\\
31	0\\
32	0\\
33	0\\
34	0\\
35	0\\
36	0\\
37	0\\
38	0\\
39	0\\
40	0\\
41	0\\
42	0\\
43	0\\
44	0\\
45	0\\
46	0\\
47	1\\
48	1\\
49	0\\
50	0\\
51	0\\
52	0\\
53	0\\
54	0\\
55	0\\
56	0\\
57	0\\
58	0\\
59	0\\
60	0\\
61	0\\
62	0\\
63	0\\
64	0\\
65	0\\
66	0\\
67	0\\
68	0\\
69	0\\
70	0\\
71	1\\
72	1\\
73	0\\
74	0\\
75	0\\
76	0\\
77	0\\
78	0\\
79	0\\
80	0\\
81	0\\
82	0\\
83	0\\
84	0\\
85	0\\
86	1\\
87	1\\
88	0\\
89	0\\
90	0\\
91	0\\
92	0\\
93	0\\
94	0\\
95	0\\
96	0\\
97	0\\
98	0\\
99	0\\
100	0\\
101	0\\
102	0\\
103	1\\
104	1\\
105	0\\
106	0\\
107	0\\
108	0\\
109	0\\
110	0\\
111	0\\
112	0\\
113	0\\
114	0\\
115	0\\
116	0\\
117	0\\
118	0\\
119	0\\
120	0\\
121	0\\
122	0\\
123	0\\
124	1\\
125	1\\
126	0\\
127	0\\
128	0\\
129	0\\
130	0\\
131	0\\
132	0\\
133	0\\
134	0\\
135	0\\
136	0\\
137	0\\
138	0\\
139	0\\
140	0\\
141	0\\
142	0\\
143	0\\
144	1\\
145	1\\
146	0\\
147	0\\
148	0\\
149	0\\
150	0\\
151	0\\
152	0\\
153	0\\
154	0\\
155	0\\
156	0\\
157	0\\
158	0\\
159	0\\
160	0\\
161	0\\
162	0\\
163	0\\
164	0\\
165	1\\
166	1\\
167	0\\
168	0\\
169	0\\
170	0\\
171	0\\
172	0\\
173	0\\
174	0\\
175	0\\
176	0\\
177	0\\
178	0\\
179	0\\
180	0\\
181	0\\
182	0\\
183	0\\
184	0\\
185	0\\
186	0\\
187	0\\
188	0\\
189	0\\
190	0\\
191	0\\
192	0\\
193	0\\
194	0\\
195	0\\
196	0\\
197	0\\
198	0\\
199	0\\
200	1\\
};

\end{axis}
\end{tikzpicture}%
}
  \subfigure[]{
\definecolor{mycolor1}{rgb}{0.00000,0.44700,0.74100}%
\definecolor{mycolor2}{rgb}{0.85000,0.32500,0.09800}%
\definecolor{mycolor3}{rgb}{0.92900,0.69400,0.12500}%
\definecolor{mycolor4}{rgb}{0.49400,0.18400,0.55600}%
\definecolor{mycolor5}{rgb}{0.46600,0.67400,0.18800}%
\begin{tikzpicture}

\begin{axis}[%
width=4.521in,
height=3.566in/1.3,
at={(0.758in,0.481in)},
scale only axis,
xmin=0,
xmax=200,
ymode=log,
ymin=1e-30,
ymax=1e5,
yminorticks=true,
ytick={1e-30,1e-20,1e-10,1},
xlabel={Time},
ylabel={State norm},
ylabel shift=7pt,
xlabel near ticks,
ylabel near ticks,
axis background/.style={fill=white},
legend style={legend cell align=left, align=left, draw=white!15!black}
]
\addplot [color=mycolor1]
  table[row sep=crcr]{%
1	1000\\
2	719.825650272431\\
3	826.654635836452\\
4	480.125983230907\\
5	364.773657342021\\
6	272.405723441645\\
7	110.8628227282\\
8	68.7703263614127\\
9	65.1504579671684\\
10	47.9072669334307\\
11	29.3058964157484\\
12	21.4107005181385\\
13	16.857780451966\\
14	12.1976264452988\\
15	8.56267851945478\\
16	6.28685675660362\\
17	4.66418857762326\\
18	3.38676284612038\\
19	2.45074036173532\\
20	1.7923388301598\\
21	1.31183616122295\\
22	0.955308218680331\\
23	0.695683222704915\\
24	0.50783693540251\\
25	0.68967997938494\\
26	0.409125210018248\\
27	0.341312541993849\\
28	0.132022845760445\\
29	0.150503134755819\\
30	0.0858834599535407\\
31	0.0536127650383885\\
32	0.0324831152855582\\
33	0.0199953103550021\\
34	0.0122269477565289\\
35	0.00749982429013249\\
36	0.00459398582420338\\
37	0.002815772141036\\
38	0.00172537956998843\\
39	0.00105736799819505\\
40	0.000647952603087075\\
41	0.000397073868967123\\
42	0.000243329314999405\\
43	0.000149114467672402\\
44	9.13785225107378e-05\\
45	5.59975376664505e-05\\
46	3.43157525284107e-05\\
47	2.10289803303124e-05\\
48	2.0784522607209e-05\\
49	1.57032085628174e-05\\
50	1.53998198572536e-05\\
51	2.2814137997114e-05\\
52	9.06409764567052e-06\\
53	8.17013373088786e-06\\
54	6.3500671580191e-06\\
55	4.05327298734949e-06\\
56	2.51306839311172e-06\\
57	1.48056304804065e-06\\
58	8.57221772372823e-07\\
59	4.86780112435396e-07\\
60	2.73862114631071e-07\\
61	1.52773245496344e-07\\
62	8.4795182098447e-08\\
63	4.68705602928895e-08\\
64	2.58369006022426e-08\\
65	1.42120924513884e-08\\
66	7.80601954323866e-09\\
67	4.28264895142553e-09\\
68	2.34770481325973e-09\\
69	1.28620891589492e-09\\
70	7.04348513722443e-10\\
71	3.8558579074276e-10\\
72	5.5292710696184e-10\\
73	3.15238046404835e-10\\
74	4.32672601628946e-10\\
75	2.97483536070084e-10\\
76	1.44208905724439e-10\\
77	1.37384970585307e-10\\
78	8.01515657697728e-11\\
79	7.23088689400232e-11\\
80	4.63606884585057e-11\\
81	3.96990100110409e-11\\
82	2.68945358894141e-11\\
83	2.21690136590144e-11\\
84	1.55545472802772e-11\\
85	1.24855017518939e-11\\
86	8.96708199229481e-12\\
87	6.8432448635697e-12\\
88	8.9560103157835e-12\\
89	1.12847940516873e-11\\
90	8.22380865938345e-12\\
91	7.6237007422885e-12\\
92	2.79229236083491e-12\\
93	1.41110861384756e-12\\
94	9.95813220344988e-13\\
95	6.60377098763436e-13\\
96	3.51030300540622e-13\\
97	2.07413741672865e-13\\
98	1.35013465594176e-13\\
99	8.00949610369151e-14\\
100	4.71406661459351e-14\\
101	2.84658475546433e-14\\
102	1.78590835776945e-14\\
103	1.07731447386769e-14\\
104	1.43850083661143e-14\\
105	1.3668892460625e-14\\
106	1.39975898287852e-14\\
107	1.06859639032919e-14\\
108	1.35939945065252e-15\\
109	7.54836741383384e-16\\
110	6.38500404098833e-16\\
111	4.54787387576655e-16\\
112	3.25101489658756e-16\\
113	2.00120444685662e-16\\
114	1.21380555784221e-16\\
115	7.0416870602293e-17\\
116	4.03761745096605e-17\\
117	2.27736924983427e-17\\
118	1.27537314914733e-17\\
119	7.09097361033794e-18\\
120	3.92649788802011e-18\\
121	2.16663691520128e-18\\
122	1.19284972999674e-18\\
123	6.55550424533835e-19\\
124	3.59821753295367e-19\\
125	2.86862008025712e-19\\
126	1.29343571220603e-19\\
127	1.19666362151485e-19\\
128	7.98565055043292e-20\\
129	1.23631409539482e-19\\
130	6.95569866173842e-20\\
131	4.35430294644581e-20\\
132	2.63119841984975e-20\\
133	1.6211335640921e-20\\
134	9.90838987831476e-21\\
135	6.07887061314121e-21\\
136	3.7232420658732e-21\\
137	2.2821646211534e-21\\
138	1.39838269829908e-21\\
139	8.56980958728439e-22\\
140	5.25153929496264e-22\\
141	3.21821771115445e-22\\
142	1.97214215841686e-22\\
143	1.20854747182631e-22\\
144	7.40607314213418e-23\\
145	7.31997879547891e-23\\
146	7.3547777847303e-23\\
147	6.89907022348619e-23\\
148	5.53149104458798e-23\\
149	5.22717452460782e-24\\
150	3.98393084335806e-24\\
151	3.27613944050609e-24\\
152	2.27465150603544e-24\\
153	1.63974689783198e-24\\
154	1.00314418754494e-24\\
155	6.0931888973928e-25\\
156	3.52952086838162e-25\\
157	2.02422203998216e-25\\
158	1.14119773796868e-25\\
159	6.39080460103453e-26\\
160	3.5526255233823e-26\\
161	1.96711573799025e-26\\
162	1.08537345104103e-26\\
163	5.97537070942213e-27\\
164	3.28375250639595e-27\\
165	1.80236696926558e-27\\
166	2.58537085217434e-27\\
167	1.55476596510418e-27\\
168	1.63724581116593e-27\\
169	1.85058142937846e-27\\
170	1.4680350455725e-27\\
171	1.26576818588808e-27\\
172	9.43804353253215e-28\\
173	7.38282207278046e-28\\
174	5.52691314712496e-28\\
175	4.22170768683378e-28\\
176	3.17600507192221e-28\\
177	2.40992179762463e-28\\
178	1.81736905076436e-28\\
179	1.37605227797023e-28\\
180	1.03881917346442e-28\\
181	7.85918634473667e-29\\
182	5.93608252949924e-29\\
183	4.48927334283836e-29\\
184	3.39165571515446e-29\\
185	2.5644914281827e-29\\
186	1.93777335784184e-29\\
187	1.4650088767191e-29\\
188	1.10709175840781e-29\\
189	8.36926186226e-30\\
190	6.32496372096004e-30\\
191	4.78122265236644e-30\\
192	3.61350364794985e-30\\
193	2.73145468898762e-30\\
194	2.06441173024304e-30\\
195	1.56045484237103e-30\\
196	1.17940331547537e-30\\
197	8.9147649461552e-31\\
198	6.73793996921848e-31\\
199	5.09295234769518e-31\\
200	3.84938260103913e-31\\
};

\end{axis}
\end{tikzpicture}%
}
  \caption{Unknown switched linear system with $5$ states, $3$ inputs,
    and $ 5$ modes evolving under the data-driven online switched
    feedback controller of Algorithm~\ref{alg:controller_1}.
    Simulations correspond to the case with noiseless data and the
    controller knowing when the system switches between modes. Plot (a) shows the
    switching signals of the system ($\sigma$) and of the controller
    ($\sigma_d$). Meanwhile, (b) plots the active (mode detection or
    stabilization) phase of the controller across time. Lastly, (c) shows the
    semi-logarithmic plot of the norm of the state as a function of time.
  }\label{fig:example}
\end{figure*}

\begin{figure*}[ht!]
  \tikzset{every picture/.style={scale=0.65}}
  \centering
  \subfigure[]{% This file was created by matlab2tikz.
\definecolor{mycolor1}{rgb}{0.00000,0.44700,0.74100}%
\definecolor{mycolor2}{rgb}{0.85000,0.32500,0.09800}%
\begin{tikzpicture}

\begin{axis}[%
width=4.521in,
height=3.566in/1.3,
at={(0.758in,0.481in)},
scale only axis,
xmin=0,
xmax=200,
ymin=0,
ymax=6,
ytick ={1,2,3,4,5},
xlabel={Time},
ylabel={Mode},
xlabel near ticks,
ylabel near ticks,
axis background/.style={fill=white},
legend style={at={(axis cs:190,1.5)},anchor=north east}
]

\addplot[const plot, color=mycolor1] table[row sep=crcr] {%
1	2\\
2	2\\
3	2\\
4	2\\
5	2\\
6	2\\
7	2\\
8	2\\
9	2\\
10	2\\
11	2\\
12	2\\
13	2\\
14	2\\
15	2\\
16	2\\
17	2\\
18	2\\
19	2\\
20	2\\
21	2\\
22	2\\
23	2\\
24	5\\
25	5\\
26	5\\
27	5\\
28	5\\
29	5\\
30	5\\
31	5\\
32	5\\
33	5\\
34	5\\
35	5\\
36	5\\
37	5\\
38	5\\
39	5\\
40	5\\
41	5\\
42	5\\
43	5\\
44	5\\
45	5\\
46	5\\
47	4\\
48	4\\
49	4\\
50	4\\
51	4\\
52	4\\
53	4\\
54	4\\
55	4\\
56	4\\
57	4\\
58	4\\
59	4\\
60	4\\
61	4\\
62	4\\
63	4\\
64	4\\
65	4\\
66	4\\
67	4\\
68	4\\
69	4\\
70	4\\
71	3\\
72	3\\
73	3\\
74	3\\
75	3\\
76	3\\
77	3\\
78	3\\
79	3\\
80	3\\
81	3\\
82	3\\
83	3\\
84	3\\
85	3\\
86	1\\
87	1\\
88	1\\
89	1\\
90	1\\
91	1\\
92	1\\
93	1\\
94	1\\
95	1\\
96	1\\
97	1\\
98	1\\
99	1\\
100	1\\
101	1\\
102	1\\
103	4\\
104	4\\
105	4\\
106	4\\
107	4\\
108	4\\
109	4\\
110	4\\
111	4\\
112	4\\
113	4\\
114	4\\
115	4\\
116	4\\
117	4\\
118	4\\
119	4\\
120	4\\
121	4\\
122	4\\
123	4\\
124	5\\
125	5\\
126	5\\
127	5\\
128	5\\
129	5\\
130	5\\
131	5\\
132	5\\
133	5\\
134	5\\
135	5\\
136	5\\
137	5\\
138	5\\
139	5\\
140	5\\
141	5\\
142	5\\
143	5\\
144	4\\
145	4\\
146	4\\
147	4\\
148	4\\
149	4\\
150	4\\
151	4\\
152	4\\
153	4\\
154	4\\
155	4\\
156	4\\
157	4\\
158	4\\
159	4\\
160	4\\
161	4\\
162	4\\
163	4\\
164	4\\
165	3\\
166	3\\
167	3\\
168	3\\
169	3\\
170	3\\
171	3\\
172	3\\
173	3\\
174	3\\
175	3\\
176	3\\
177	3\\
178	3\\
179	3\\
180	3\\
181	3\\
182	3\\
183	3\\
184	3\\
185	3\\
186	3\\
187	3\\
188	3\\
189	3\\
190	3\\
191	3\\
192	3\\
193	3\\
194	3\\
195	3\\
196	3\\
197	3\\
198	3\\
199	3\\
200	5\\
201	5\\
};
\addlegendentry{$\sigma(t)$}

\addplot[const plot, color=mycolor2] table[row sep=crcr] {%
1	1\\
2	2\\
3	2\\
4	2\\
5	2\\
6	2\\
7	2\\
8	2\\
9	2\\
10	2\\
11	2\\
12	2\\
13	2\\
14	2\\
15	2\\
16	2\\
17	2\\
18	2\\
19	2\\
20	2\\
21	2\\
22	2\\
23	2\\
24	2\\
25	2\\
26	5\\
27	5\\
28	5\\
29	5\\
30	5\\
31	5\\
32	5\\
33	5\\
34	5\\
35	5\\
36	5\\
37	5\\
38	5\\
39	5\\
40	5\\
41	5\\
42	5\\
43	5\\
44	5\\
45	5\\
46	5\\
47	5\\
48	5\\
49	4\\
50	4\\
51	4\\
52	4\\
53	4\\
54	4\\
55	4\\
56	4\\
57	4\\
58	4\\
59	4\\
60	4\\
61	4\\
62	4\\
63	4\\
64	4\\
65	4\\
66	4\\
67	4\\
68	4\\
69	4\\
70	4\\
71	4\\
72	4\\
73	3\\
74	3\\
75	3\\
76	3\\
77	3\\
78	3\\
79	3\\
80	3\\
81	3\\
82	3\\
83	3\\
84	3\\
85	3\\
86	3\\
87	3\\
88	1\\
89	1\\
90	1\\
91	1\\
92	1\\
93	1\\
94	1\\
95	1\\
96	1\\
97	1\\
98	1\\
99	1\\
100	1\\
101	1\\
102	1\\
103	1\\
104	1\\
105	4\\
106	4\\
107	4\\
108	4\\
109	4\\
110	4\\
111	4\\
112	4\\
113	4\\
114	4\\
115	4\\
116	4\\
117	4\\
118	4\\
119	4\\
120	4\\
121	4\\
122	4\\
123	4\\
124	4\\
125	4\\
126	5\\
127	5\\
128	5\\
129	5\\
130	5\\
131	5\\
132	5\\
133	5\\
134	5\\
135	5\\
136	5\\
137	5\\
138	5\\
139	5\\
140	5\\
141	5\\
142	5\\
143	5\\
144	5\\
145	5\\
146	4\\
147	4\\
148	4\\
149	4\\
150	4\\
151	4\\
152	4\\
153	4\\
154	4\\
155	4\\
156	4\\
157	4\\
158	4\\
159	4\\
160	4\\
161	4\\
162	4\\
163	4\\
164	4\\
165	4\\
166	4\\
167	3\\
168	3\\
169	3\\
170	3\\
171	3\\
172	3\\
173	3\\
174	3\\
175	3\\
176	3\\
177	3\\
178	3\\
179	3\\
180	3\\
181	3\\
182	3\\
183	3\\
184	3\\
185	3\\
186	3\\
187	3\\
188	3\\
189	3\\
190	3\\
191	3\\
192	3\\
193	3\\
194	3\\
195	3\\
196	3\\
197	3\\
198	3\\
199	3\\
200	3\\
};
\addlegendentry{$\sigma_d(t)$}

\end{axis}
\end{tikzpicture}%
}
  \subfigure[]{
\definecolor{mycolor1}{rgb}{0.00000,0.44700,0.74100}%
\begin{tikzpicture}

\begin{axis}[%
width=4.521in,
height=3.566in/1.3,
at={(0.758in,0.481in)},
scale only axis,
xmin=0,
xmax=200,
ymin=-0.5,
ymax=1.5,
ytick={0,1},
yticklabels={S,MD},
xlabel={Time},
ylabel={Phase},
xlabel near ticks,
ylabel near ticks,
axis background/.style={fill=white},
legend style={legend cell align=left, align=left, draw=white!15!black}
]
\addplot[const plot, color=mycolor1] table[row sep=crcr] {%
1	1\\
2	0\\
3	0\\
4	0\\
5	0\\
6	0\\
7	0\\
8	0\\
9	0\\
10	0\\
11	0\\
12	0\\
13	0\\
14	0\\
15	0\\
16	0\\
17	0\\
18	0\\
19	0\\
20	0\\
21	0\\
22	0\\
23	0\\
24	1\\
25	1\\
26	0\\
27	0\\
28	0\\
29	0\\
30	0\\
31	0\\
32	0\\
33	0\\
34	0\\
35	0\\
36	0\\
37	0\\
38	0\\
39	0\\
40	0\\
41	0\\
42	0\\
43	0\\
44	0\\
45	0\\
46	0\\
47	1\\
48	1\\
49	0\\
50	0\\
51	0\\
52	0\\
53	0\\
54	0\\
55	0\\
56	0\\
57	0\\
58	0\\
59	0\\
60	0\\
61	0\\
62	0\\
63	0\\
64	0\\
65	0\\
66	0\\
67	0\\
68	0\\
69	0\\
70	0\\
71	1\\
72	1\\
73	0\\
74	0\\
75	0\\
76	0\\
77	0\\
78	0\\
79	0\\
80	0\\
81	0\\
82	0\\
83	0\\
84	0\\
85	0\\
86	1\\
87	1\\
88	0\\
89	0\\
90	0\\
91	0\\
92	0\\
93	0\\
94	0\\
95	0\\
96	0\\
97	0\\
98	0\\
99	0\\
100	0\\
101	0\\
102	0\\
103	1\\
104	1\\
105	0\\
106	0\\
107	0\\
108	0\\
109	0\\
110	0\\
111	0\\
112	0\\
113	0\\
114	0\\
115	0\\
116	0\\
117	0\\
118	0\\
119	0\\
120	0\\
121	0\\
122	0\\
123	0\\
124	1\\
125	1\\
126	0\\
127	0\\
128	0\\
129	0\\
130	0\\
131	0\\
132	0\\
133	0\\
134	0\\
135	0\\
136	0\\
137	0\\
138	0\\
139	0\\
140	0\\
141	0\\
142	0\\
143	0\\
144	1\\
145	1\\
146	0\\
147	0\\
148	0\\
149	0\\
150	0\\
151	0\\
152	0\\
153	0\\
154	0\\
155	0\\
156	0\\
157	0\\
158	0\\
159	0\\
160	0\\
161	0\\
162	0\\
163	0\\
164	0\\
165	1\\
166	1\\
167	0\\
168	0\\
169	0\\
170	0\\
171	0\\
172	0\\
173	0\\
174	0\\
175	0\\
176	0\\
177	0\\
178	0\\
179	0\\
180	0\\
181	0\\
182	0\\
183	0\\
184	0\\
185	0\\
186	0\\
187	0\\
188	0\\
189	0\\
190	0\\
191	0\\
192	0\\
193	0\\
194	0\\
195	0\\
196	0\\
197	0\\
198	0\\
199	0\\
200	1\\
};

\end{axis}
\end{tikzpicture}%
}
  \subfigure[]{
\definecolor{mycolor1}{rgb}{0.00000,0.44700,0.74100}%
\definecolor{mycolor2}{rgb}{0.85000,0.32500,0.09800}%
\definecolor{mycolor3}{rgb}{0.92900,0.69400,0.12500}%
\definecolor{mycolor4}{rgb}{0.49400,0.18400,0.55600}%
\definecolor{mycolor5}{rgb}{0.46600,0.67400,0.18800}%
\begin{tikzpicture}

\begin{axis}[%
width=4.521in,
height=3.566in/1.3,
at={(0.758in,0.481in)},
scale only axis,
xmin=0,
xmax=200,
ymode=log,
ymin=1e-30,
ymax=1e5,
yminorticks=true,
ytick={1e-30,1e-20,1e-10,1},
xlabel={Time},
ylabel={State norm},
ylabel shift=7pt,
xlabel near ticks,
ylabel near ticks,
axis background/.style={fill=white},
legend style={legend cell align=left, align=left, draw=white!15!black}
]
\addplot [color=mycolor1]
  table[row sep=crcr]{%
1	1000\\
2	760.688064512933\\
3	771.132265061172\\
4	508.679758206459\\
5	379.386941441333\\
6	289.509350700353\\
7	124.760651031162\\
8	69.0451047294562\\
9	66.1728245122776\\
10	50.3307007408891\\
11	30.7586359201081\\
12	21.944318128895\\
13	17.3593768307039\\
14	12.6785687797479\\
15	8.8776244898947\\
16	6.48843812914195\\
17	4.82163585108039\\
18	3.50802146451094\\
19	2.53615263590217\\
20	1.85317814209426\\
21	1.35704904209297\\
22	0.988604905339854\\
23	0.71973780613035\\
24	0.52531345003943\\
25	0.713519409729402\\
26	0.474279043839899\\
27	0.249594196190693\\
28	0.088944909804909\\
29	0.0669733630021888\\
30	0.0405340124576429\\
31	0.0248232777051384\\
32	0.0152198544388704\\
33	0.00932485241724303\\
34	0.00571489193337747\\
35	0.00350198331034376\\
36	0.00214608588137062\\
37	0.00131512790632602\\
38	0.00080592447508812\\
39	0.000493876365334856\\
40	0.000302651788689149\\
41	0.000185467471151437\\
42	0.00011365602919254\\
43	6.96493563348407e-05\\
44	4.2681707045876e-05\\
45	2.61557051934647e-05\\
46	1.60284340973352e-05\\
47	9.82235788650718e-06\\
48	9.70817490447424e-06\\
49	7.73791684748625e-06\\
50	9.64937090063346e-06\\
51	1.22019876438282e-05\\
52	4.28454525022286e-06\\
53	4.07789628883101e-06\\
54	3.17736912881401e-06\\
55	2.05340715482169e-06\\
56	1.2745695023317e-06\\
57	7.53990510922662e-07\\
58	4.37017416592529e-07\\
59	2.48532478567754e-07\\
60	1.39913273893738e-07\\
61	7.81011136321284e-08\\
62	4.33650339142082e-08\\
63	2.39775809100615e-08\\
64	1.32200646156105e-08\\
65	7.27312769625362e-09\\
66	3.99522426927452e-09\\
67	2.19210337632106e-09\\
68	1.20176213489397e-09\\
69	6.5842520484224e-10\\
70	3.60576165274732e-10\\
71	1.97397309909608e-10\\
72	2.83067070436416e-10\\
73	1.52001077159496e-10\\
74	1.89016063471986e-10\\
75	1.46448850828574e-10\\
76	6.22521377092676e-11\\
77	6.84870953715121e-11\\
78	3.07768247387797e-11\\
79	3.34744210088252e-11\\
80	1.58949344511205e-11\\
81	1.68724925607867e-11\\
82	8.73894047616738e-12\\
83	8.78456461042296e-12\\
84	4.99360873096434e-12\\
85	4.70375207939275e-12\\
86	2.89513701041948e-12\\
87	3.06089921060689e-12\\
88	2.84068644879952e-12\\
89	2.63057735682656e-12\\
90	1.43617832301586e-12\\
91	1.46025671225999e-12\\
92	6.26381644414496e-13\\
93	3.1187668212097e-13\\
94	2.0130950826513e-13\\
95	1.41039106131694e-13\\
96	7.69958266264172e-14\\
97	4.53307296841971e-14\\
98	2.83698665429776e-14\\
99	1.71793110387657e-14\\
100	1.02270331239837e-14\\
101	6.20236816852978e-15\\
102	3.84956798852906e-15\\
103	2.30789466596434e-15\\
104	2.81445984367035e-15\\
105	2.596544484858e-15\\
106	1.94898522721954e-15\\
107	1.53152852485794e-15\\
108	3.57402384572363e-16\\
109	3.16536131498398e-16\\
110	2.80309284999677e-16\\
111	1.79055481953183e-16\\
112	1.14403364856714e-16\\
113	6.76823607911122e-17\\
114	3.95759391879581e-17\\
115	2.25442892783405e-17\\
116	1.27330129541177e-17\\
117	7.11643833196715e-18\\
118	3.95696167425354e-18\\
119	2.18953913115175e-18\\
120	1.20803008965306e-18\\
121	6.64891760206325e-19\\
122	3.65361357245566e-19\\
123	2.00514505334333e-19\\
124	1.09947010971014e-19\\
125	8.76465345643627e-20\\
126	4.03780053047915e-20\\
127	4.30879757026165e-20\\
128	3.00166559210322e-20\\
129	4.59231836643234e-20\\
130	2.5943609374493e-20\\
131	1.62319235806166e-20\\
132	9.81512113678236e-21\\
133	6.04596660679676e-21\\
134	3.69572992232482e-21\\
135	2.26724765601619e-21\\
136	1.38869578686369e-21\\
137	8.51193846587101e-22\\
138	5.2156628771019e-22\\
139	3.19634583307149e-22\\
140	1.95870758810971e-22\\
141	1.20032327864129e-22\\
142	7.35565083209142e-23\\
143	4.50761241574706e-23\\
144	2.76230013895764e-23\\
145	2.73018887064255e-23\\
146	2.64843561171624e-23\\
147	2.83831327712905e-23\\
148	2.46654977897762e-23\\
149	2.54802703577896e-24\\
150	1.82313532884347e-24\\
151	1.76451730733547e-24\\
152	1.21160870030867e-24\\
153	8.44870047971214e-25\\
154	5.14451151837338e-25\\
155	3.09543381473871e-25\\
156	1.78769244188441e-25\\
157	1.02167044697483e-25\\
158	5.75029527588033e-26\\
159	3.21523130717071e-26\\
160	1.78570009585987e-26\\
161	9.88009397956309e-27\\
162	5.44870861492653e-27\\
163	2.99854392382179e-27\\
164	1.64739545424459e-27\\
165	9.04026502358457e-28\\
166	1.29673265631346e-27\\
167	9.23882007066845e-28\\
168	8.55939153529936e-28\\
169	8.97794024417912e-28\\
170	8.30223291241501e-28\\
171	6.93832914258379e-28\\
172	5.41131337913786e-28\\
173	4.14012232920596e-28\\
174	3.16088919895275e-28\\
175	2.38289756395152e-28\\
176	1.81195281401604e-28\\
177	1.36375125686045e-28\\
178	1.03523614841479e-28\\
179	7.79714349024204e-29\\
180	5.91179277491676e-29\\
181	4.45674326151081e-29\\
182	3.37605509414942e-29\\
183	2.54704041085011e-29\\
184	1.92815281869761e-29\\
185	1.45548846056079e-29\\
186	1.10131196860441e-29\\
187	8.316656926854e-30\\
188	6.29081636877524e-30\\
189	4.7518781047263e-30\\
190	3.5935481759545e-30\\
191	2.71497121591369e-30\\
192	2.05283403375221e-30\\
193	1.55114925408581e-30\\
194	1.17271882282198e-30\\
195	8.86204369441418e-31\\
196	6.69947366320787e-31\\
197	5.06300768300149e-31\\
198	3.82729681074529e-31\\
199	2.89254017722404e-31\\
200	2.18648671194686e-31\\
};

\end{axis}
\end{tikzpicture}%
}
  \caption{Simulation results for the system considered in
    Figure~\ref{fig:example} with the same noiseless data, but with a
    switching signal completely unknown to the
    controller, requiring the use of \eqref{eq:switch detection} to detect switches.}\label{fig:example1}
\end{figure*}
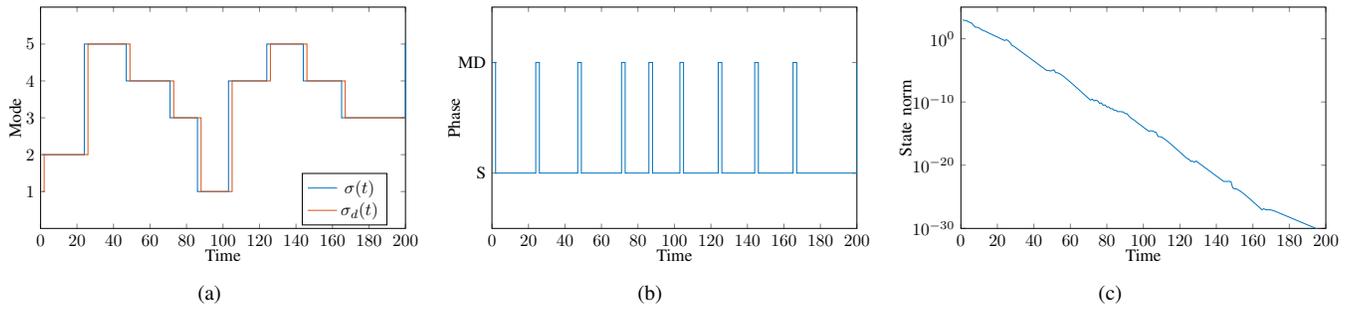

\begin{figure*}[ht!]
  \tikzset{every picture/.style={scale=0.65}}
  \centering
  \subfigure[]{% This file was created by matlab2tikz.
\definecolor{mycolor1}{rgb}{0.00000,0.44700,0.74100}%
\definecolor{mycolor2}{rgb}{0.85000,0.32500,0.09800}%
\begin{tikzpicture}

\begin{axis}[%
width=4.521in,
height=3.566in/1.3,
at={(0.758in,0.481in)},
scale only axis,
xmin=0,
xmax=200,
ymin=0,
ymax=6,
ytick ={1,2,3,4,5},
xlabel={Time},
ylabel={Mode},
xlabel near ticks,
ylabel near ticks,
axis background/.style={fill=white},
legend style={at={(axis cs:190,1.5)},anchor=north east}
]

\addplot[const plot, color=mycolor1] table[row sep=crcr] {%
1	2\\
2	2\\
3	2\\
4	2\\
5	2\\
6	2\\
7	2\\
8	2\\
9	2\\
10	2\\
11	2\\
12	2\\
13	2\\
14	2\\
15	2\\
16	2\\
17	2\\
18	2\\
19	2\\
20	2\\
21	2\\
22	2\\
23	2\\
24	5\\
25	5\\
26	5\\
27	5\\
28	5\\
29	5\\
30	5\\
31	5\\
32	5\\
33	5\\
34	5\\
35	5\\
36	5\\
37	5\\
38	5\\
39	5\\
40	5\\
41	5\\
42	5\\
43	5\\
44	5\\
45	5\\
46	5\\
47	4\\
48	4\\
49	4\\
50	4\\
51	4\\
52	4\\
53	4\\
54	4\\
55	4\\
56	4\\
57	4\\
58	4\\
59	4\\
60	4\\
61	4\\
62	4\\
63	4\\
64	4\\
65	4\\
66	4\\
67	4\\
68	4\\
69	4\\
70	4\\
71	3\\
72	3\\
73	3\\
74	3\\
75	3\\
76	3\\
77	3\\
78	3\\
79	3\\
80	3\\
81	3\\
82	3\\
83	3\\
84	3\\
85	3\\
86	1\\
87	1\\
88	1\\
89	1\\
90	1\\
91	1\\
92	1\\
93	1\\
94	1\\
95	1\\
96	1\\
97	1\\
98	1\\
99	1\\
100	1\\
101	1\\
102	1\\
103	4\\
104	4\\
105	4\\
106	4\\
107	4\\
108	4\\
109	4\\
110	4\\
111	4\\
112	4\\
113	4\\
114	4\\
115	4\\
116	4\\
117	4\\
118	4\\
119	4\\
120	4\\
121	4\\
122	4\\
123	4\\
124	5\\
125	5\\
126	5\\
127	5\\
128	5\\
129	5\\
130	5\\
131	5\\
132	5\\
133	5\\
134	5\\
135	5\\
136	5\\
137	5\\
138	5\\
139	5\\
140	5\\
141	5\\
142	5\\
143	5\\
144	4\\
145	4\\
146	4\\
147	4\\
148	4\\
149	4\\
150	4\\
151	4\\
152	4\\
153	4\\
154	4\\
155	4\\
156	4\\
157	4\\
158	4\\
159	4\\
160	4\\
161	4\\
162	4\\
163	4\\
164	4\\
165	3\\
166	3\\
167	3\\
168	3\\
169	3\\
170	3\\
171	3\\
172	3\\
173	3\\
174	3\\
175	3\\
176	3\\
177	3\\
178	3\\
179	3\\
180	3\\
181	3\\
182	3\\
183	3\\
184	3\\
185	3\\
186	3\\
187	3\\
188	3\\
189	3\\
190	3\\
191	3\\
192	3\\
193	3\\
194	3\\
195	3\\
196	3\\
197	3\\
198	3\\
199	3\\
200	5\\
201	5\\
};
\addlegendentry{$\sigma(t)$}

\addplot[const plot, color=mycolor2] table[row sep=crcr] {%
1	1\\
2	1\\
3	2\\
4	2\\
5	2\\
6	2\\
7	2\\
8	2\\
9	2\\
10	2\\
11	2\\
12	2\\
13	2\\
14	2\\
15	2\\
16	2\\
17	2\\
18	2\\
19	2\\
20	2\\
21	2\\
22	2\\
23	2\\
24	2\\
25	2\\
26	2\\
27	2\\
28	5\\
29	5\\
30	5\\
31	5\\
32	5\\
33	5\\
34	5\\
35	5\\
36	5\\
37	5\\
38	5\\
39	5\\
40	5\\
41	5\\
42	5\\
43	5\\
44	5\\
45	5\\
46	5\\
47	5\\
48	5\\
49	5\\
50	5\\
51	5\\
52	5\\
53	5\\
54	4\\
55	4\\
56	4\\
57	4\\
58	4\\
59	4\\
60	4\\
61	4\\
62	4\\
63	4\\
64	4\\
65	4\\
66	4\\
67	4\\
68	4\\
69	4\\
70	4\\
71	4\\
72	4\\
73	4\\
74	4\\
75	4\\
76	4\\
77	3\\
78	3\\
79	3\\
80	3\\
81	3\\
82	3\\
83	3\\
84	3\\
85	3\\
86	3\\
87	3\\
88	3\\
89	3\\
90	3\\
91	1\\
92	1\\
93	1\\
94	1\\
95	1\\
96	1\\
97	1\\
98	1\\
99	1\\
100	1\\
101	1\\
102	1\\
103	1\\
104	1\\
105	1\\
106	1\\
107	1\\
108	4\\
109	4\\
110	4\\
111	4\\
112	4\\
113	4\\
114	4\\
115	4\\
116	4\\
117	4\\
118	4\\
119	4\\
120	4\\
121	4\\
122	4\\
123	4\\
124	4\\
125	4\\
126	4\\
127	4\\
128	4\\
129	5\\
130	5\\
131	5\\
132	5\\
133	5\\
134	5\\
135	5\\
136	5\\
137	5\\
138	5\\
139	5\\
140	5\\
141	5\\
142	5\\
143	5\\
144	5\\
145	5\\
146	5\\
147	5\\
148	5\\
149	5\\
150	5\\
151	5\\
152	4\\
153	4\\
154	4\\
155	4\\
156	4\\
157	4\\
158	4\\
159	4\\
160	4\\
161	4\\
162	4\\
163	4\\
164	4\\
165	4\\
166	4\\
167	4\\
168	4\\
169	4\\
170	3\\
171	3\\
172	3\\
173	3\\
174	3\\
175	3\\
176	3\\
177	3\\
178	3\\
179	3\\
180	3\\
181	3\\
182	3\\
183	3\\
184	3\\
185	3\\
186	3\\
187	3\\
188	3\\
189	3\\
190	3\\
191	3\\
192	3\\
193	3\\
194	3\\
195	3\\
196	3\\
197	3\\
198	3\\
199	3\\
200	3\\
};
\addlegendentry{$\sigma_d(t)$}

\end{axis}
\end{tikzpicture}%
}
  \subfigure[]{
\definecolor{mycolor1}{rgb}{0.00000,0.44700,0.74100}%
\begin{tikzpicture}

\begin{axis}[%
width=4.521in,
height=3.566in/1.3,
at={(0.758in,0.481in)},
scale only axis,
xmin=0,
xmax=200,
ymin=-0.5,
ymax=1.5,
ytick={0,1},
yticklabels={S,MD},
xlabel={Time},
ylabel={Phase},
xlabel near ticks,
ylabel near ticks,
axis background/.style={fill=white},
legend style={legend cell align=left, align=left, draw=white!15!black}
]
\addplot[const plot, color=mycolor1] table[row sep=crcr] {%
1	1\\
2	1\\
3	0\\
4	0\\
5	0\\
6	0\\
7	0\\
8	0\\
9	0\\
10	0\\
11	0\\
12	0\\
13	0\\
14	0\\
15	0\\
16	0\\
17	0\\
18	0\\
19	0\\
20	0\\
21	0\\
22	0\\
23	0\\
24	0\\
25	1\\
26	1\\
27	1\\
28	0\\
29	0\\
30	0\\
31	0\\
32	0\\
33	0\\
34	0\\
35	0\\
36	0\\
37	0\\
38	0\\
39	0\\
40	0\\
41	0\\
42	0\\
43	0\\
44	0\\
45	0\\
46	0\\
47	0\\
48	1\\
49	1\\
50	1\\
51	1\\
52	1\\
53	1\\
54	0\\
55	0\\
56	0\\
57	0\\
58	0\\
59	0\\
60	0\\
61	0\\
62	0\\
63	0\\
64	0\\
65	0\\
66	0\\
67	0\\
68	0\\
69	0\\
70	0\\
71	0\\
72	1\\
73	1\\
74	1\\
75	1\\
76	1\\
77	0\\
78	0\\
79	0\\
80	0\\
81	0\\
82	0\\
83	0\\
84	0\\
85	0\\
86	0\\
87	1\\
88	1\\
89	1\\
90	1\\
91	0\\
92	0\\
93	0\\
94	0\\
95	0\\
96	0\\
97	0\\
98	0\\
99	0\\
100	0\\
101	0\\
102	0\\
103	0\\
104	1\\
105	1\\
106	1\\
107	1\\
108	0\\
109	0\\
110	0\\
111	0\\
112	0\\
113	0\\
114	0\\
115	0\\
116	0\\
117	0\\
118	0\\
119	0\\
120	0\\
121	0\\
122	0\\
123	0\\
124	0\\
125	1\\
126	1\\
127	1\\
128	1\\
129	0\\
130	0\\
131	0\\
132	0\\
133	0\\
134	0\\
135	0\\
136	0\\
137	0\\
138	0\\
139	0\\
140	0\\
141	0\\
142	0\\
143	0\\
144	0\\
145	1\\
146	1\\
147	1\\
148	1\\
149	1\\
150	1\\
151	1\\
152	0\\
153	0\\
154	0\\
155	0\\
156	0\\
157	0\\
158	0\\
159	0\\
160	0\\
161	0\\
162	0\\
163	0\\
164	0\\
165	0\\
166	1\\
167	1\\
168	1\\
169	1\\
170	0\\
171	0\\
172	0\\
173	0\\
174	0\\
175	0\\
176	0\\
177	0\\
178	0\\
179	0\\
180	0\\
181	0\\
182	0\\
183	0\\
184	0\\
185	0\\
186	0\\
187	0\\
188	0\\
189	0\\
190	0\\
191	0\\
192	0\\
193	0\\
194	0\\
195	0\\
196	0\\
197	0\\
198	0\\
199	0\\
200	0\\
};

\end{axis}
\end{tikzpicture}%
}
  \subfigure[]{
    \input{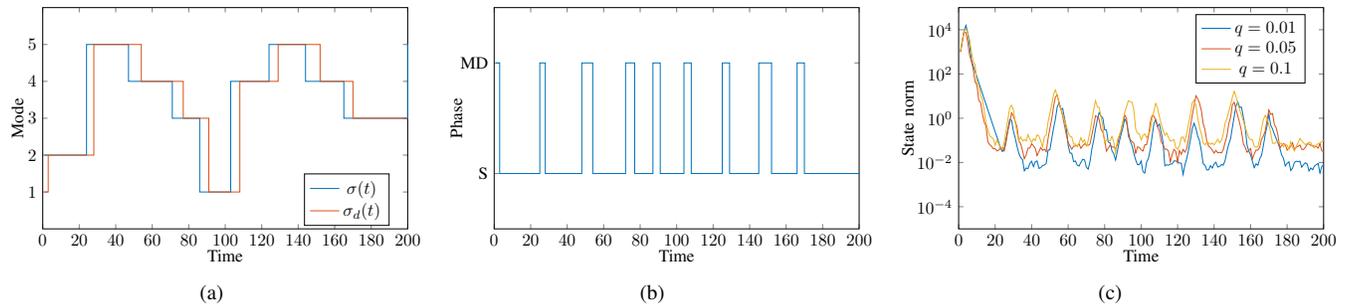}}
  \caption{Unknown switched linear system with $5$ states, $3$ inputs,
    and $ 5$ modes evolving under the data-driven online switched
    feedback controller in Algorithm~\ref{alg:controller_1}.
    Simulations correspond to the case with noisy data (with
    measurement noise bounded by $q=0.01$) and the controller knowing
    when the system switches modes. Plot (a) shows the switching
    signals of the system ($\sigma$) and of the controller
    ($\sigma_d$). In (b) we plot the active (mode detection or
    stabilization) phase of the controller across time. Lastly, (c)
    shows the semi-logarithmic plot of the state norm as a function of
    time. {\color{black} Lastly, (c) shows additional results for
      experiments subject to noise with $q=0.05$ and $q=0.1$.}
    % (c) shows the resulting state trajectories.  
  }\label{fig:example2}
\end{figure*}

\begin{figure*}[ht!]
  \tikzset{every picture/.style={scale=0.65}}
  \centering
  \subfigure[]{% This file was created by matlab2tikz.
\definecolor{mycolor1}{rgb}{0.00000,0.44700,0.74100}%
\definecolor{mycolor2}{rgb}{0.85000,0.32500,0.09800}%
\begin{tikzpicture}

\begin{axis}[%
width=4.521in,
height=3.566in/1.3,
at={(0.758in,0.481in)},
scale only axis,
xmin=0,
xmax=200,
ymin=0,
ymax=6,
ytick ={1,2,3,4,5},
xlabel={Time},
ylabel={Mode},
xlabel near ticks,
ylabel near ticks,
axis background/.style={fill=white},
legend style={at={(axis cs:190,1.5)},anchor=north east}
]

\addplot[const plot, color=mycolor1] table[row sep=crcr] {%
1	2\\
2	2\\
3	2\\
4	2\\
5	2\\
6	2\\
7	2\\
8	2\\
9	2\\
10	2\\
11	2\\
12	2\\
13	2\\
14	2\\
15	2\\
16	2\\
17	2\\
18	2\\
19	2\\
20	2\\
21	2\\
22	2\\
23	2\\
24	5\\
25	5\\
26	5\\
27	5\\
28	5\\
29	5\\
30	5\\
31	5\\
32	5\\
33	5\\
34	5\\
35	5\\
36	5\\
37	5\\
38	5\\
39	5\\
40	5\\
41	5\\
42	5\\
43	5\\
44	5\\
45	5\\
46	5\\
47	4\\
48	4\\
49	4\\
50	4\\
51	4\\
52	4\\
53	4\\
54	4\\
55	4\\
56	4\\
57	4\\
58	4\\
59	4\\
60	4\\
61	4\\
62	4\\
63	4\\
64	4\\
65	4\\
66	4\\
67	4\\
68	4\\
69	4\\
70	4\\
71	3\\
72	3\\
73	3\\
74	3\\
75	3\\
76	3\\
77	3\\
78	3\\
79	3\\
80	3\\
81	3\\
82	3\\
83	3\\
84	3\\
85	3\\
86	1\\
87	1\\
88	1\\
89	1\\
90	1\\
91	1\\
92	1\\
93	1\\
94	1\\
95	1\\
96	1\\
97	1\\
98	1\\
99	1\\
100	1\\
101	1\\
102	1\\
103	4\\
104	4\\
105	4\\
106	4\\
107	4\\
108	4\\
109	4\\
110	4\\
111	4\\
112	4\\
113	4\\
114	4\\
115	4\\
116	4\\
117	4\\
118	4\\
119	4\\
120	4\\
121	4\\
122	4\\
123	4\\
124	5\\
125	5\\
126	5\\
127	5\\
128	5\\
129	5\\
130	5\\
131	5\\
132	5\\
133	5\\
134	5\\
135	5\\
136	5\\
137	5\\
138	5\\
139	5\\
140	5\\
141	5\\
142	5\\
143	5\\
144	4\\
145	4\\
146	4\\
147	4\\
148	4\\
149	4\\
150	4\\
151	4\\
152	4\\
153	4\\
154	4\\
155	4\\
156	4\\
157	4\\
158	4\\
159	4\\
160	4\\
161	4\\
162	4\\
163	4\\
164	4\\
165	3\\
166	3\\
167	3\\
168	3\\
169	3\\
170	3\\
171	3\\
172	3\\
173	3\\
174	3\\
175	3\\
176	3\\
177	3\\
178	3\\
179	3\\
180	3\\
181	3\\
182	3\\
183	3\\
184	3\\
185	3\\
186	3\\
187	3\\
188	3\\
189	3\\
190	3\\
191	3\\
192	3\\
193	3\\
194	3\\
195	3\\
196	3\\
197	3\\
198	3\\
199	3\\
200	5\\
201	5\\
};
\addlegendentry{$\sigma(t)$}

\addplot[const plot, color=mycolor2] table[row sep=crcr] {%
1	1\\
2	1\\
3	2\\
4	2\\
5	2\\
6	2\\
7	2\\
8	2\\
9	2\\
10	2\\
11	2\\
12	2\\
13	2\\
14	2\\
15	2\\
16	2\\
17	2\\
18	2\\
19	2\\
20	2\\
21	2\\
22	2\\
23	2\\
24	2\\
25	2\\
26	2\\
27	2\\
28	2\\
29	5\\
30	5\\
31	5\\
32	5\\
33	5\\
34	5\\
35	5\\
36	5\\
37	5\\
38	5\\
39	5\\
40	5\\
41	5\\
42	5\\
43	5\\
44	5\\
45	5\\
46	5\\
47	5\\
48	5\\
49	5\\
50	5\\
51	5\\
52	5\\
53	5\\
54	4\\
55	4\\
56	4\\
57	4\\
58	4\\
59	4\\
60	4\\
61	4\\
62	4\\
63	4\\
64	4\\
65	4\\
66	4\\
67	4\\
68	4\\
69	4\\
70	4\\
71	4\\
72	4\\
73	4\\
74	4\\
75	4\\
76	4\\
77	4\\
78	4\\
79	3\\
80	3\\
81	3\\
82	3\\
83	3\\
84	3\\
85	3\\
86	3\\
87	3\\
88	3\\
89	1\\
90	1\\
91	1\\
92	1\\
93	1\\
94	1\\
95	1\\
96	1\\
97	1\\
98	1\\
99	1\\
100	1\\
101	1\\
102	1\\
103	1\\
104	1\\
105	1\\
106	1\\
107	1\\
108	1\\
109	1\\
110	4\\
111	4\\
112	4\\
113	4\\
114	4\\
115	4\\
116	4\\
117	4\\
118	4\\
119	4\\
120	4\\
121	4\\
122	4\\
123	4\\
124	4\\
125	4\\
126	4\\
127	4\\
128	4\\
129	4\\
130	4\\
131	4\\
132	5\\
133	5\\
134	5\\
135	5\\
136	5\\
137	5\\
138	5\\
139	5\\
140	5\\
141	5\\
142	5\\
143	5\\
144	5\\
145	5\\
146	5\\
147	5\\
148	5\\
149	5\\
150	5\\
151	5\\
152	5\\
153	5\\
154	5\\
155	5\\
156	4\\
157	4\\
158	4\\
159	4\\
160	4\\
161	4\\
162	4\\
163	4\\
164	4\\
165	4\\
166	4\\
167	4\\
168	4\\
169	4\\
170	4\\
171	3\\
172	3\\
173	3\\
174	3\\
175	3\\
176	3\\
177	3\\
178	3\\
179	3\\
180	3\\
181	3\\
182	3\\
183	3\\
184	3\\
185	3\\
186	3\\
187	3\\
188	3\\
189	3\\
190	3\\
191	3\\
192	3\\
193	3\\
194	3\\
195	3\\
196	3\\
197	3\\
198	3\\
199	3\\
200	3\\
};
\addlegendentry{$\sigma_d(t)$}

\end{axis}
\end{tikzpicture}%
}
  \subfigure[]{
\definecolor{mycolor1}{rgb}{0.00000,0.44700,0.74100}%
\begin{tikzpicture}

\begin{axis}[%
width=4.521in,
height=3.566in/1.3,
at={(0.758in,0.481in)},
scale only axis,
xmin=0,
xmax=200,
ymin=-0.5,
ymax=1.5,
ytick={0,1},
yticklabels={S,MD},
xlabel={Time},
ylabel={Phase},
xlabel near ticks,
ylabel near ticks,
axis background/.style={fill=white},
legend style={legend cell align=left, align=left, draw=white!15!black}
]
\addplot[const plot, color=mycolor1] table[row sep=crcr] {%
1	1\\
2	1\\
3	0\\
4	0\\
5	0\\
6	0\\
7	0\\
8	0\\
9	0\\
10	0\\
11	0\\
12	0\\
13	0\\
14	0\\
15	0\\
16	0\\
17	0\\
18	0\\
19	0\\
20	0\\
21	0\\
22	0\\
23	0\\
24	0\\
25	1\\
26	1\\
27	1\\
28	1\\
29	0\\
30	0\\
31	0\\
32	0\\
33	0\\
34	0\\
35	0\\
36	0\\
37	0\\
38	0\\
39	0\\
40	0\\
41	0\\
42	0\\
43	0\\
44	0\\
45	0\\
46	0\\
47	0\\
48	0\\
49	0\\
50	0\\
51	0\\
52	1\\
53	1\\
54	0\\
55	0\\
56	0\\
57	0\\
58	0\\
59	0\\
60	0\\
61	0\\
62	0\\
63	0\\
64	0\\
65	0\\
66	0\\
67	0\\
68	0\\
69	0\\
70	0\\
71	0\\
72	0\\
73	0\\
74	1\\
75	1\\
76	1\\
77	1\\
78	1\\
79	0\\
80	0\\
81	0\\
82	0\\
83	0\\
84	0\\
85	0\\
86	0\\
87	1\\
88	1\\
89	0\\
90	0\\
91	0\\
92	0\\
93	0\\
94	0\\
95	0\\
96	0\\
97	0\\
98	0\\
99	0\\
100	0\\
101	0\\
102	0\\
103	0\\
104	0\\
105	0\\
106	0\\
107	1\\
108	1\\
109	1\\
110	0\\
111	0\\
112	0\\
113	0\\
114	0\\
115	0\\
116	0\\
117	0\\
118	0\\
119	0\\
120	0\\
121	0\\
122	0\\
123	0\\
124	0\\
125	0\\
126	0\\
127	0\\
128	0\\
129	1\\
130	1\\
131	1\\
132	0\\
133	0\\
134	0\\
135	0\\
136	0\\
137	0\\
138	0\\
139	0\\
140	0\\
141	0\\
142	0\\
143	0\\
144	0\\
145	0\\
146	0\\
147	0\\
148	0\\
149	0\\
150	0\\
151	1\\
152	1\\
153	1\\
154	1\\
155	1\\
156	0\\
157	0\\
158	0\\
159	0\\
160	0\\
161	0\\
162	0\\
163	0\\
164	0\\
165	0\\
166	1\\
167	1\\
168	1\\
169	1\\
170	1\\
171	0\\
172	0\\
173	0\\
174	0\\
175	0\\
176	0\\
177	0\\
178	0\\
179	0\\
180	0\\
181	0\\
182	0\\
183	0\\
184	0\\
185	0\\
186	0\\
187	0\\
188	0\\
189	0\\
190	0\\
191	0\\
192	0\\
193	0\\
194	0\\
195	0\\
196	0\\
197	0\\
198	0\\
199	0\\
200	0\\
};
\end{axis}
\end{tikzpicture}%
}
  \subfigure[]{
    \input{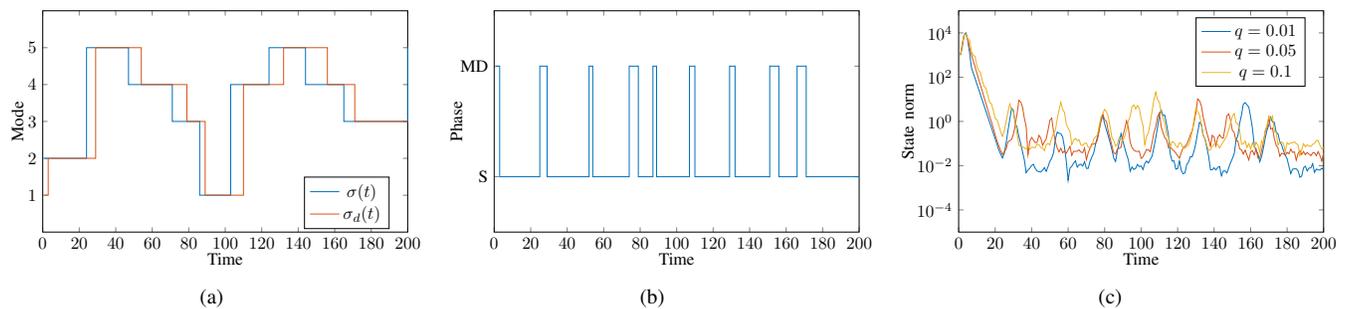}}
  \caption{Simulation results for the systems considered in
    Figure~\ref{fig:example2} with the same noisy data, but with a
    switching signal completely unknown to the controller, requiring the use of \eqref{eq:switch detection} to detect switches.
  }\label{fig:example3}
\end{figure*}

\section{Simulation results}\label{sec:simulation}

We illustrate the performance of the proposed data-driven switching
feedback controller through 4 simulation experiments of an unknown
switched linear system with $n=5$ states, $m=3$ inputs, and $p=5$
modes. The system matrices $\hat{A}_i, \hat{B}_i$ are randomly
  selected and given by
\footnotesize\begin{align*}
\hat{A}_1&\!=\!\begin{bmatrix}
	-0.73   \hspace{-1em}&-0.68   \hspace{-1em}&-0.03   \hspace{-1em}&-1.34   \hspace{-1em}&-0.20\\
	-0.26    \hspace{-1em}&\hphantom{-}0.27    \hspace{-1em}&\hphantom{-}0.18   \hspace{-1em}&-0.02    \hspace{-1em}&\hphantom{-}0.79\\
	-0.00   \hspace{-1em}&-0.15   \hspace{-1em}&-0.09   \hspace{-1em}&-0.13    \hspace{-1em}&\hphantom{-}0.46\\
	\hphantom{-}0.55    \hspace{-1em}&\hphantom{-}0.01   \hspace{-1em}& \hphantom{-}0.47    \hspace{-1em}&\hphantom{-}0.85    \hspace{-1em}&\hphantom{-}0.42\\
	-0.24    \hspace{-1em}&\hphantom{-}0.38   \hspace{-1em}&-0.17    \hspace{-1em}&\hphantom{-}0.81    \hspace{-1em}&\hphantom{-}0.05
\end{bmatrix} \hspace{-1em}&\hspace{-1em}
\hat{B}_1&\!=\!\begin{bmatrix}
	-0.38    \hspace{-1em}&\hphantom{-}0.56    \hspace{-1em}&\hphantom{-}0.70\\
	-0.36   \hspace{-1em}&-0.91   \hspace{-1em}&-0.81\\
	\hphantom{-}0.42   \hspace{-1em}&-1.15    \hspace{-1em}&\hphantom{-}0.57\\
	\hphantom{-}0.54   \hspace{-1em}&-0.52    \hspace{-1em}&\hphantom{-}1.69\\
	\hphantom{-}0.38   \hspace{-1em}&-0.80   \hspace{-1em}&-0.88
\end{bmatrix} \hspace{-1em} \\
\hat{A}_2&\!=\!\begin{bmatrix}
	\hphantom{-}0.26   \hspace{-1em}&-0.03    \hspace{-1em}&\hphantom{-}0.67    \hspace{-1em}&\hphantom{-}0.77   \hspace{-1em}&-0.00\\
	-0.02    \hspace{-1em}&\hphantom{-}0.46    \hspace{-1em}&\hphantom{-}0.10    \hspace{-1em}&\hphantom{-}0.48    \hspace{-1em}&\hphantom{-}0.54\\
	-0.15   \hspace{-1em}&-0.07   \hspace{-1em}&-0.17    \hspace{-1em}&\hphantom{-}0.51   \hspace{-1em}&-0.15\\
	\hphantom{-}0.43   \hspace{-1em}&-0.37   \hspace{-1em}&-1.01   \hspace{-1em}&-0.36    \hspace{-1em}&\hphantom{-}0.32\\
	\hphantom{-}0.00    \hspace{-1em}&\hphantom{-}0.21    \hspace{-1em}&\hphantom{-}0.90    \hspace{-1em}&\hphantom{-}0.11    \hspace{-1em}&\hphantom{-}0.12
\end{bmatrix} \hspace{-1em}&\hspace{-1em}
\hat{B}_2&\!=\!\begin{bmatrix}
	\hphantom{-}1.59    \hspace{-1em}&\hphantom{-}1.26   \hspace{-1em}&-1.04\\
	\hphantom{-}0.73    \hspace{-1em}&\hphantom{-}1.62   \hspace{-1em}&-0.42\\
	\hphantom{-}0.96   \hspace{-1em}&-0.54    \hspace{-1em}&\hphantom{-}0.73\\
	\hphantom{-}0.85    \hspace{-1em}&\hphantom{-}0.90    \hspace{-1em}&\hphantom{-}1.51\\
	-0.12   \hspace{-1em}&-0.78    \hspace{-1em}&\hphantom{-}0.00
\end{bmatrix} \hspace{-1em}\\
\hat{A}_3&\!=\!\begin{bmatrix}
	-0.25    \hspace{-1em}&\hphantom{-}0.52    \hspace{-1em}&\hphantom{-}0.39    \hspace{-1em}&1.15   \hspace{-1em}&-0.29\\
	\hphantom{-}0.29    \hspace{-1em}&\hphantom{-}0.28   \hspace{-1em}&-0.21    \hspace{-1em}&\hphantom{-}0.18    \hspace{-1em}&\hphantom{-}0.36\\
	-0.19   \hspace{-1em}&-0.40   \hspace{-1em}&-0.16    \hspace{-1em}&1.19   \hspace{-1em}&-0.08\\
	-0.03    \hspace{-1em}&\hphantom{-}0.72   \hspace{-1em}&-0.80    \hspace{-1em}&\hphantom{-}0.21   \hspace{-1em}&-0.66\\
	\hphantom{-}0.19    \hspace{-1em}&\hphantom{-}0.09   \hspace{-1em}&-0.08    \hspace{-1em}&\hphantom{-}0.00    \hspace{-1em}&\hphantom{-}0.22
\end{bmatrix} \hspace{-1em}&\hspace{-1em}
\hat{B}_3&\!=\!\begin{bmatrix}
	\hphantom{-}2.04   \hspace{-1em}&-0.29    \hspace{-1em}&\hphantom{-}0.10\\
	\hphantom{-}0.20   \hspace{-1em}&-2.00   \hspace{-1em}&-0.19\\
	-0.12   \hspace{-1em}&-2.62    \hspace{-1em}&\hphantom{-}0.50\\
	\hphantom{-}0.76    \hspace{-1em}&\hphantom{-}0.71    \hspace{-1em}&\hphantom{-}0.34\\
	-2.52    \hspace{-1em}&\hphantom{-}0.01    \hspace{-1em}&\hphantom{-}0.94
\end{bmatrix} \hspace{-1em}\\
\hat{A}_4&\!=\!\begin{bmatrix}
	\hphantom{-}0.18    \hspace{-1em}&\hphantom{-}0.22    \hspace{-1em}&\hphantom{-}0.16   \hspace{-1em}&-0.20    \hspace{-1em}&\hphantom{-}0.64\\
	-0.07    \hspace{-1em}&\hphantom{-}0.58    \hspace{-1em}&\hphantom{-}0.99   \hspace{-1em}&-0.12    \hspace{-1em}&\hphantom{-}0.66\\
	\hphantom{-}0.41    \hspace{-1em}&\hphantom{-}0.27    \hspace{-1em}&\hphantom{-}0.24    \hspace{-1em}&\hphantom{-}0.40   \hspace{-1em}&-0.36\\
	\hphantom{-}0.36   \hspace{-1em}&-0.33    \hspace{-1em}&\hphantom{-}0.80   \hspace{-1em}&-0.05    \hspace{-1em}&\hphantom{-}0.35\\
	\hphantom{-}0.04    \hspace{-1em}&\hphantom{-}0.09   \hspace{-1em}&-0.83   \hspace{-1em}&-0.21    \hspace{-1em}&\hphantom{-}0.04
\end{bmatrix}\hspace{-1em}& \hspace{-1em}
\hat{B}_4&\!=\!\begin{bmatrix}
	\hphantom{-}0.47    \hspace{-1em}&\hphantom{-}0.95   \hspace{-1em}&-0.41\\
	-1.07    \hspace{-1em}&\hphantom{-}0.37    \hspace{-1em}&\hphantom{-}0.53\\
	-0.50    \hspace{-1em}&\hphantom{-}1.14    \hspace{-1em}&\hphantom{-}1.77\\
	-0.61    \hspace{-1em}&\hphantom{-}0.12   \hspace{-1em}&-1.93\\
	-0.99    \hspace{-1em}&\hphantom{-}1.88    \hspace{-1em}&\hphantom{-}2.02
\end{bmatrix} \hspace{-1em}\\
\hat{A}_5&\!=\!\begin{bmatrix}
	\hphantom{-}0.27    \hspace{-1em}&\hphantom{-}0.09    \hspace{-1em}&\hphantom{-}0.87    \hspace{-1em}&\hphantom{-}0.28   \hspace{-1em}&-0.10\\
	-0.54    \hspace{-1em}&\hphantom{-}0.44    \hspace{-1em}&\hphantom{-}0.25    \hspace{-1em}&\hphantom{-}0.35    \hspace{-1em}&\hphantom{-}0.04\\
	\hphantom{-}0.53   \hspace{-1em}&-0.37    \hspace{-1em}&\hphantom{-}0.00    \hspace{-1em}&\hphantom{-}0.36   \hspace{-1em}&-0.49\\
	\hphantom{-}0.16    \hspace{-1em}&\hphantom{-}0.20    \hspace{-1em}&\hphantom{-}0.00    \hspace{-1em}&\hphantom{-}0.67    \hspace{-1em}&\hphantom{-}0.26\\
	-0.10   \hspace{-1em}&-0.42   \hspace{-1em}&-0.08   \hspace{-1em}&-0.48   \hspace{-1em}&-0.14
\end{bmatrix}&\hspace{-1em}
\hat{B}_5&\!=\!\begin{bmatrix}
	-0.86    \hspace{-1em}&\hphantom{-}0.58   \hspace{-1em}&-1.38\\
	\hphantom{-}1.52   \hspace{-1em}&-0.54   \hspace{-1em}&-2.46\\
	-0.74   \hspace{-1em}&-1.00    \hspace{-1em}&\hphantom{-}0.16\\
	-2.31    \hspace{-1em}&\hphantom{-}1.47   \hspace{-1em}&-0.55\\
	\hphantom{-}0.44   \hspace{-1em}&-1.91    \hspace{-1em}&\hphantom{-}0.16
\end{bmatrix}
\end{align*}\normalsize

For all simulations, we select the same switching signal $\sigma$, shown as
the blue curve in Figure~\ref{fig:example}(a) through
Figure~\ref{fig:example3}(a). On average, this signal has 1 switch per 20 
units of time. Each of the 4 simulations has as initial condition
$x(0)=\begin{bmatrix} 1000 &0&0&0&0
\end{bmatrix}^\top$.

In the first two simulations, a noiseless data pair $(U^i_-,X^i)$ with
$T^i=7$ is collected in the initialization step for each mode
$i \in \calP = \{1,\dots,5\}$. Note that $n+m=8$, and hence it is 
impossible to uniquely determine the dynamics of the modes from these 7
measurements. We set $\lambda=\sqrt{0.8}$ as the desired decay rate for each mode
and note that both Assumptions~\ref{ass:ass_on_data-I}
and~\ref{ass:ass_on_data-II} are satisfied on the initial
data. Moreover, with $c=0.1$, it can be computed that 
%$\mu=1.26$,
$\mu=1.12$,
(see \eqref{def:mu}), and the LMI~\eqref{LMI_for_MD} holds with
%$\lambda_u=5.86$. 
$\lambda_u=2.42$. 
We apply our proposed data-driven controller: in the
first simulation, shown in Figure~\ref{fig:example}, the controller is
aware of when the systems switches and in the second simulation,
shown in Figure~\ref{fig:example1}, the switching signal is unknown. In
both cases, we run the mode detection algorithm for noiseless data,
that is Algorithm~\ref{alg:mode_detection}.

Comparing Figure~\ref{fig:example}(a) and
Figure~\ref{fig:example1}(a), we observe that the switching signals
$\sigma_d$ of the controller generated in the two different scenarios
are almost the same, meaning that the condition \eqref{what_to_obey}
with $q=0$ is effective in detecting switches. When $\sigma_d(t)$ is
not equal to $\sigma(t)$, the controller is operating in the mode
detection phase, as seen in Figure~\ref{fig:example}(b) and
Figure~\ref{fig:example1}(b). One can see that for both simulations,
$\sigma_d$ switches at the same rate as $\sigma$; in particular, when
$\sigma$ is unknown in the second simulation, switches are always
immediately detected, implying that the controller changes to the mode
detection phase whenever $\sigma$ changes value. Moreover, each
instance of the mode detection phase takes 2 units of time.

This also implies that Assumption~\ref{ass:detection_after_exercution}
holds with parameters $\tau=20$ and $\eta=0.1$.  As a result,
\begin{equation*}
	    \left(1-\frac{\ln\lambda_u}{\ln\lambda}\right)\eta +
	\left(1-\frac{\ln\mu}{\ln\lambda}\right)\frac{1}{\tau}=0.9942<1.
\end{equation*}
Therefore, the condition~\eqref{ADT_AAT_condition} is met and hence,
by Theorem~\ref{thm:stability_control}, the closed-loop
system~\eqref{def:sys} satisfies~\eqref{ISS-estimate}. Since the data
is noiseless ($q=0$), the system is actually asymptotically
stable. This is indeed reflected by the plots of the norms of the
trajectories, as shown in Figures~\ref{fig:example}(c)
and~\ref{fig:example1}(c).

In the third and fourth simulations, we assume the measurements are
corrupted with noise of magnitude $q=0.01$. During the initialization
step, a noisy data pair $(U^i_-,X^i)$ with $T^i=9$ is collected for
each mode.
%The larger $T$ ensures the feasibility
%of~\eqref{eqn:common_K_LMI} in Theorem~\ref{thm:common_K}. Empirically
%speaking, due to the presence of the noise $w$, we need more
%measurements of the data to ``better'' know the system.
% and we need a less conservative LMI~\eqref{eqn:common_K_LMI} in
% order for Assumptions~\ref{ass:ass_on_data-I} and
% \ref{ass:ass_on_data-II} to hold.
The desired decay rate is again $\lambda=\sqrt{0.8}$ and it can be verified
that both Assumptions~\ref{ass:ass_on_data-I}
and~\ref{ass:ass_on_data-II} are satisfied on the initial
data. Moreover, with $c=1$, it can be computed that 
%$\mu=2.14$, 
$\mu=1.46$, 
and the LMI~\eqref{LMI_for_MD} holds with
%$\lambda_u=264$.
$\lambda_u=16.2$.
We remark that a value of $c$ larger than in the
noiseless simulations is needed in order to distinguish the inputs
from the noise. Similar to the previous two examples, we study the
performance of the data-driven controller in the scenarios with
knowledge of when the system mode switches,
shown in Figure~\ref{fig:example2}, and when the switching signal is
unknown, shown in Figure~\ref{fig:example3}. As suggested by
Figures~\ref{fig:example2}(b) and~\ref{fig:example3}(b), the presence
of noise increases the duration of the mode detection phase. In both
simulations, we obtain $\eta= 0.22$.

Comparing Figures~\ref{fig:example3}(a) and~\ref{fig:example3}(b), one
can see that the delay between $\sigma$ and $\sigma_d$ is not equal to
the length of the mode detection phase. This difference precisely
corresponds to the time required to detect the switch (note that
instantaneous convergence of the states is still guaranteed by
\eqref{what_to_obey}). In this case, we have
\begin{equation*}
  \left(1-\frac{\ln\lambda_u}{\ln\lambda}\right)\eta +
  \left(1-\frac{\ln\mu}{\ln\lambda}\right)\frac{1}{\tau}= 5.94>1,
\end{equation*}
i.e., the condition~\eqref{ADT_AAT_condition} is not met.
Nevertheless, the behavior displayed by the trajectories in
Figures~\ref{fig:example2}(c) and~\ref{fig:example3}(c), with
oscillations (that are bounded) in the stabilization phase due to the
noise, still suggests practical exponential stability. {\color{black}In
  addition, we also plot the norms of the state resulting from the
  same initial condition and switching signal, but with noise bounds
  $q=0.05$ and $0.1$. One can roughly see that the the oscillation
  becomes larger as $q$ increases.}

\section{Conclusions}\label{sec:conclusion}
We have considered the problem of stabilizing an unknown switched
linear system on the basis of measured data. Prior to the online
operation of the switched system, we have access to measurements of
each of the individual modes. We have derived conditions in terms of
linear matrix inequalities under which this pre-collected data is
informative enough for finding a uniformly stabilizing controller for
each mode. Once the system is running, we have access to online
measurements of the currently active mode. Our online switched
controller design alternates between a phase that performs mode
detection on the basis of the online measurements and a stabilization
phase that exploits this identification. Under average dwell- and
activation-time assumptions on the switching signal, the proposed
controller guarantees a {\color{black}practical} stability property of the
closed-loop switched system. Our technical exposition has dealt with
both noiseless and noisy measurements, and the cases when the
controller knows the switching times of the system or has complete
lack of knowledge about them. Future research will investigate
  methods of input design in the presence of noise in order to
  accelerate the mode detection phase. Another open problem is to
  incorporate the online measurements to either speed up the rate of
  convergence or shrink the uncertainty regarding the dynamics of the
  modes. Lastly, employing clustering methods, or otherwise exploiting
  the online measurements, might enable the detection of switches even
  during the mode detection phase.

\bibliography{../bib/alias,../bib/JC,../bib/Main,../bib/Main-add,../bib/New,../bib/FB}
\bibliographystyle{IEEEtran}

\appendix

\section*{Proof of Theorem~\ref{thm:stability_control}}
The proof of Theorem~\ref{thm:stability_control} consists of two
  steps. Firstly, we construct two discrete timers to encode the ADT
  and AAT conditions. We then use a similar approach as
  in~\cite{GY-DL:15,JIP-ART:17} to show that a discrete-time system
  whose state consists of the state of the closed-loop system, the
  switching signal and the two timers is practically exponentially
  stable, which further implies that the closed-loop system itself has
  the same stability property.

% for handling these conditions is a common approach in the study of
% switched systems and in fact similar timers can also be found
% in~\cite{GY-DL:15,JIP-ART:17}.  For the ADT condition, the timer
% $\tau_d$ provides a mechanism for keeping track of to what extent
% the system behavior is in line with the assumption that we have
% $1/\tau$ switches per unit time.  Similarly, for the AAT condition,
% the timer $\tau_a$ provides a mechanism for keeping track of to what
% extent the system behavior is in line with the assumption of
% spending a fraction $\eta$ of the total time in the mode detection
% phase.  Intuitively speaking, each switch will reduce the
% monotonically increasing value of a timer $\tau_d$ by a fixed
% number; switches are not allowed if this reduction causes $\tau_d$
% to be negative.  By doing so, the ADT
% condition~\eqref{ADT_condition} on the switching signal is
% ensured. Similarly, each time the controller dwells in the mode
% detection phase will reduce the monotonically increasing value of a
% timer $\tau_a$ by a fixed number; further dwelling in the mode
% detection phase is not allowed if this reduction causes $\tau_a$ to
% be negative. By doing so, the AAT condition~\eqref{AAT_condition} on
% the switching signal is ensured.

\begin{lemma}[Discrete timers for ADT and AAT
  conditions]\label{lem:timers}
  Assume the system does not switch while the controller is in the
  mode detection phase.  Then,
  Assumption~\ref{ass:detection_after_exercution}\ref{itm:asump adt}
  implies that there exists a timer $\tau_d:\N\mapsto[0,N_0]$ such
  that
  \begin{subequations}\label{def:tau_d}
    \begin{align}
      \tau_d(t+1)&
                   =\tau_d(t)+\frac{1}{\tau}-1
      &&\mbox{ if } t\in
         \bT^m,\label{def:tau_d_jump}
      \\ 
      \tau_d(t+1)&\in\big[\tau_d(t),\tau_d(t)+\frac{1}{\tau}\big]
      &&\mbox{ otherwise}.\label{def:tau_d_flow}
    \end{align}
  \end{subequations}
  Similarly
  Assumption~\ref{ass:detection_after_exercution}\ref{itm:asump aat}
  implies that there is a timer $\tau_a:\N\mapsto[0,N_0]$ such that
  \begin{subequations}\label{def:tau_a}
    \begin{align}
      \tau_a(t+1)&=\tau_a(t)+\eta-1&&\mbox{ if
                                      }t\in\{t_i^m,\cdots,t_i^s-1\},i\in\N,\label{def:tau_a_unstable}
      \\ 
      \tau_a(t+1)&\in\big[\tau_a(t),\tau_a(t)+\eta]&&\mbox{
                                                      otherwise}.\label{def:tau_a_stable} 
    \end{align}
  \end{subequations}
\end{lemma}
\begin{proof}%[Proof of Lemma~\ref{lem:timers}]
  We study the ADT condition first. Let
  \begin{equation}
    n_d(t):=\min_{s=0,\ldots,t} \{\calN(0,s)-\frac{s}{\tau}\}.
  \end{equation}
  By definition $n_d(t)$ is non-increasing. We claim that
  \[\tau_d(t):=N_0+n_d(t)-(\calN(0,t)-\frac{t}{\tau})\]
  satisfies \eqref{def:tau_d}. Clearly, it follows from the definition
  of $n_d(t)$ that $\tau_d(t)\leq N_0$.  Also, for $s=0,\ldots,t$, we
  have
  \begin{multline*}
    N_0+(\calN(0,s)-\frac{s}{\tau})-(\calN(0,t)-\frac{t}{\tau})
    \\
    =\calN(0,s)+\big(\frac{t-s}{\tau}+N_0\big)-\calN(0,t)
    \\
    \geq \calN(0,s)+\calN(s,t)-\calN(0,t)=0.
  \end{multline*}
  Hence the range of $\tau_d$ is $[0,N_0]$.  To verify that
  \eqref{def:tau_d_jump} holds, let $t\in\bT^m$. In this case, we have
  $\calN(0,t+1)=\calN(0,t)+1$. Since $\tau\geq 1$,
  \[
    \calN(0,t+1)-\frac{t+1}{\tau}\geq \calN(0,t)-\frac{t}{\tau}\geq n_d(t).
  \]
  This implies that $n_d(t+1)=n_d(t)$, and hence 
  \begin{multline*}
    \tau_d(t+1)-\tau_d(t)
    \\
    =(n_d(t+1)-n_d(t))-(\calN(0,t+1)-\calN(0,t)-\frac{1}{\tau})=\frac{1}{\tau}-1,
  \end{multline*}
  concluding~\eqref{def:tau_d_jump}. We now consider the
  case where $t\not\in\bT^m$. Note that by definition
  $\calN(0,t+1)=\calN(0,t)$. Direct inspection yields
  \begin{equation}\label{adt_one_side}
    \tau_d(t+1)-\tau_d(t)=(n_d(t+1)-n_d(t))+\frac{1}{\tau}\leq\frac{1}{\tau}.
  \end{equation}
  Meanwhile, note that $n_d(t+1)<n_d(t)$ if and only if
  $\calN(0,t+1)-\frac{t+1}{\tau}< n_d(t)$, in which case
  \begin{multline*}
    n_d(t+1)-n_d(t)=\calN(0,t+1)-\frac{t+1}{\tau}- n_d(t)\\
    \geq
    (\calN(0,t+1)-\frac{t+1}{\tau})-(\calN(0,t)-\frac{t}{\tau})=-\frac{1}{\tau} ,
  \end{multline*}
  and consequently,
  \begin{equation}\label{adt_another_side}
    \tau_d(t+1)-\tau_d(t)\geq 0 .
  \end{equation}
  Therefore, \eqref{def:tau_d_flow} follows from~\eqref{adt_one_side}
  and~\eqref{adt_another_side}.
  
  To study the AAT condition, we define
  \begin{equation}
    n_a(t):=\min_{s=0,\ldots,t} \{ \calM(0,s)-\eta s \}.
  \end{equation}
  By definition $n_a(t)$ is non-increasing. We claim that
  \[
    \tau_a(t):=T_0+n_a(t)-(\calM(0,t)-\eta t)
  \]
  satisfies \eqref{def:tau_a}. Clearly it follows from the definition
  of $n_a(t)$ that $\tau_a(t)\leq T_0$. Moreover, for $s=0,\ldots,t$,
  \begin{multline*}
    T_0+(\calM(0,s)-\eta s)-(\calM(0,t)-\eta t)\\
    =\calM(0,s)+(\eta(t-s)+T_0)-\calM(0,t)\\
    \geq \calM(0,s)+\calM(s,t)-\calM(0,t)=0.
  \end{multline*}
  Hence the range of $\tau_a$ is indeed $[0,T_0]$. To verify
  \eqref{def:tau_a_unstable}, we first assume
  $t\in\{t_i^m,\cdots,t_i^s-1\}$ for some $i\in\N$. In this case we
  have $\calM(0,t+1)=\calM(0,t)+1$. Since $\eta\in[0,1]$, we have that
  $\calM(0,t+1)-\eta(t+1)\geq \calM(0,t)-\eta t$. In turn, this implies that
  $n_a(t+1)=n_a(t)$. Therefore we can conclude
  \begin{multline*}
    \tau_a(t+1)-\tau_a(t)
    \\
    =(n_a(t+1)-n_a(t))-(\calM(0,t+1))-\calM(0,t)-\eta)=\eta-1,
  \end{multline*}
  proving~\eqref{def:tau_a_unstable}. We move our
  attention to the case where $t\in\{t_i^s,\cdots,t_{i+1}^m-1\}$ for
  some $i\in\N$. We have $\calN(0,t+1)=\calN(0,t)$ and hence
  \begin{equation}\label{aat_one_side}
    \tau_a(t+1)-\tau_a(t)=(n_a(t+1)-n_a(t))+\eta\leq \eta.
  \end{equation}
  On the other hand, note that $n_a(t+1)<n_a(t)$ only if
  $\calM(0,t+1)-\eta(t+1)<n_a(t)$. Therefore, we can conclude
  \begin{multline*}
    n_a(t+1)-n_a(t)=\calM(0,t+1)-\eta(t+1)-n_a(t)\\
    \geq (\calM(0,t+1)-\eta(t+1))-(\calM(0,t)-\eta t)=-\eta.
  \end{multline*}
  In turn, this implies that
  \begin{equation}\label{aat_another_side}
    \tau_a(t+1)-\tau_a(t)\geq 0.
  \end{equation}
  We can now conclude \eqref{def:tau_a_stable} by combining
  \eqref{aat_one_side} and \eqref{aat_another_side}.
\end{proof}

We rely on Lemma~\ref{lem:timers} to prove 
Theorem~\ref{thm:stability_control} next.

\begin{proof}[Proof of Theorem~\ref{thm:stability_control}]
  Construct two timers $\tau_d:\N\mapsto[0,N_0]$,
  $\tau_a:\N\mapsto[0,T_0]$ as in Lemma~\ref{lem:timers}.  Denote the
  extended state at time~$t$,
  \begin{equation}
    \xi(t):=\begin{pmatrix}
      x(t)\\
      \sigma_d(t)\\
      \tau_d(t)\\
      \tau_a(t)
    \end{pmatrix}\in\R^n\times\calP\times[0,N_0]\times[0,T_0]=:\calX.
  \end{equation}
  Further define functions $U,V,W:\calX\mapsto \R_{\geq 0}$ by
  \begin{align*}
    U(\xi)&:={\Big(\frac{\mu}{\lambda}\Big)^{\tau_d}
            \Big(\frac{\lambda_u}{\lambda}\Big)^{\tau_a}},    
    \\
    V(\xi)&:=V_{\sigma_d}(x)=\norm{P_{\sigma_d}^{\frac{1}{2}}x},
    \\
    W(\xi)&:=U(\xi)V(\xi) .
  \end{align*}
  Note that since $\mu\geq 1$, $\lambda_u\geq 1$ and $0<\lambda<1$, we
  have that $U$ is increasing with respect to both $\tau_d$ and
  $\tau_a$, and $a\geq {\lambda}$, where $a$ is defined in
  \eqref{def:a}. In addition, since
  $\tau_d\in[0,N_0],\tau_a\in[0,T_0]$, we see that
  \[
    U(\xi)\in\left[1,{
        \Big(\frac{\mu}{\lambda}\Big)^{N_0}\Big(\frac{\lambda_u}{\lambda}\Big)^{T_0}}\right]
    = [1,b],
  \]
  where $b$ is defined in \eqref{def:b}.  It follows from
  \eqref{ADT_AAT_condition} that
  \begin{align*}
    \ln
    a &= \left(\ln\lambda+(\ln\mu-\ln\lambda)\frac{1}{\tau}+(\ln\lambda_u-\ln\lambda)\eta\right) 
    \\
      &= {\ln\lambda}
        \left(1-\big(1-\frac{\ln\mu}{\ln\lambda}\big)\frac{1}{\tau}-\big(1-\frac{\ln\lambda_u}{\ln\lambda}\big)\eta\right)<0.  
  \end{align*}
  Therefore we have that $a<1$. We now investigate the one-step change
  of the function $t \mapsto W(\xi(t))$ in three separate cases.

  1) The case $t\in \bT^m$; that is, a switch is detected at
  time~$t$. In this case, the controller has switched to the mode
  detection phase.  We apply the Schur complement on
  \eqref{LMI_for_MD} to deduce that
  \begin{equation*}
    \begin{bmatrix}
      \lambda_u^2 P_i&0\\0&0
    \end{bmatrix}-\begin{bmatrix}
      \hat{A}_i&\hat{B}_i
    \end{bmatrix}^\top P_i\begin{bmatrix}
      \hat{A}_i&\hat{B}_i
    \end{bmatrix}-k_i\begin{bmatrix}
      c^2I_n&0\\0&-I_n
    \end{bmatrix}\succeq 0.
  \end{equation*}
  Since $k_i\geq 0$, for any $x\in\R^n,u\in\R^m$ such that
  \begin{equation}\label{S_procedure_condition}
    \begin{bmatrix}
      x\\u
    \end{bmatrix}^\top\begin{bmatrix}
      c^2I_n&0\\0&-I_n
    \end{bmatrix}\begin{bmatrix}
      x\\u
    \end{bmatrix}\succeq 0,
  \end{equation}
  it must be that 
  \begin{equation}\label{S_procedure_result}
    \begin{bmatrix}
      x\\u
    \end{bmatrix}^\top\left(\begin{bmatrix}
        \lambda_u^2 P_i&0\\0&0
      \end{bmatrix}-\begin{bmatrix}
	\hat{A}_i&\hat{B}_i
      \end{bmatrix}^\top P_i\begin{bmatrix}
        \hat{A}_i&\hat{B}_i
      \end{bmatrix}\right)\begin{bmatrix}
      x\\u
    \end{bmatrix}\succeq 0.
  \end{equation}  
  Notice that \eqref{S_procedure_condition} is equivalent to
  $\norm{u}\leq c\norm{x}$, which is the condition on the control used
  during the mode detection phase. On the other hand,
  \eqref{S_procedure_result} is equivalent to
  \begin{equation}\label{decreasing_V}
    (\hat{A}_ix+\hat{B}_i u)^\top P_i(\hat{A}_ix+\hat{B}_iu)\leq \lambda_u^2x^\top P_i x.
  \end{equation}
  Denote $\sigma_d(t)=j,\sigma_d(t+1)=k$, which are different because
  of the phase switch at $t\in\bT^m$.  It follows from the triangle
  inequality that
  \begin{align*}
    V_{\sigma_d(t+1)}
    &(x(t+1))=V_k(x(t+1))
    \\
    &=\norm{P_k^{\frac{1}{2}}(\hat{A}_kx(t)+\hat{B}_ku(t)+w(t))}
    \\
    % &\leq \norm{P_k^{\frac{1}{2}}(A_kx(t)+B_ku(t))}+\norm{P_k^{\frac{1}{2}}w(t)}\\
    % &=\sqrt{(A_kx(t)+B_k u(t))^\top P_k(A_kx(t)+B_ku(t))}+\norm{P_k^{\frac{1}{2}}w(t)}\\
    % &\leq \sqrt{\lambda_ux(t)^\top P_k x(t)}+\norm{P_k^{\frac{1}{2}}}\norm{w(t)}\\
    &\leq {\lambda_u}\norm{P_k^{\frac{1}{2}}x(t)}+p_{\max}q
    \\
    &\leq{\lambda_u} \Vert
      P_k^{\frac{1}{2}}P_j^{-\frac{1}{2}}\Vert\norm{P_j^{\frac{1}{2}}x(t)}+p_{\max}q
    \\
    &\leq{\lambda_u\mu} V_{\sigma(t)}(x(t))+p_{\max}q.
  \end{align*}
  This can be equivalently written as
  $ V(\xi(t+1))\leq {\lambda_u\mu}V(\xi(t))+p_{\max}q$.  On the
  other hand, it follows from \eqref{def:tau_d_jump} and
  \eqref{def:tau_a_unstable} that
  \begin{multline*}
    U(\xi(t+1)) =
    {\Big(\frac{\mu}{\lambda}\Big)^{\tau_d(t)+\frac{1}{\tau}-1}
      \Big(\frac{\lambda_u}{\lambda}\Big)^{\tau_a(t)+\eta-1}}
    \\
    ={\Big(\frac{\mu}{\lambda}\Big)^{\frac{1}{\tau}-1}
      \Big(\frac{\lambda_u}{\lambda}\Big)^{\eta-1}}U(\xi(t))={\frac{\lambda
      }{\lambda_u\mu}}aU(\xi(t)).
  \end{multline*}
  Therefore,
  \begin{align*}
    W(\xi(t+1))
    &=U(\xi(t+1))V(\xi(t+1))
    \\
    &
      \leq{\lambda }aU(\xi(t))V(\xi(t))+{\frac{\lambda }{
      \lambda_u\mu}}U(\xi(t))ap_{\max}q
    \\
    &\leq {\lambda }
      aW(\xi(t))+{\frac{\lambda }{
      \lambda_u\mu}}abp_{\max}q,
  \end{align*}
  and we conclude
  \begin{equation}\label{uniform_one_step_inequality_W}
    W(\xi(t+1))\leq aW(\xi(t))+\tilde bq,
  \end{equation}    
  where $\tilde b:=\frac{ab}{{\lambda}}p_{\max}$.
  
  2) The case where $t\in\{t_i^m+1,\cdots,t_i^s-1\}$ for some
  $i\in\N$. Since the controller is in the mode detection phase and
  $\sigma_d$ does not switch at $t+1$, we can similarly conclude from
  \eqref{decreasing_V} as in the previous case that
  \begin{equation*}
    V(\xi(t+1))\leq {\lambda_u}V(\xi(t))+p_{\max}q.
  \end{equation*}
  Moreover, from \eqref{def:tau_d_flow} and \eqref{def:tau_a_unstable}
  we can conclude that
  \begin{align*}
    U(\xi(t+1))
    &\leq{\Big(\frac{\mu}{\lambda}\Big)^{\tau_d(t)+\frac{1}{\tau}}
      \Big(\frac{\lambda_u}{\lambda}\Big)^{\tau_a(t)+\eta-1}}
    \\
    &={\Big(\frac{\mu}{\lambda}\Big)^{\frac{1}{\tau}}
      \Big(\frac{\lambda_u}{\lambda}\Big)^{\eta-1}}U(\xi(t)) =
      \frac{a}{{\lambda_u}}U(\xi(t)).    
  \end{align*}
  Therefore
  \begin{align*}
    W(\xi(t+1))
    &=U(\xi(t+1))V(\xi(t+1))
    \\
    &\leq  a
      U(\xi(t))V(\xi(t))+\frac{a}{{\lambda_u}}U(\xi(t))p_{\max}q
    \\
    &\leq a W(\xi(t))+\frac{ab }{{\lambda_u}}p_{\max}q ,
  \end{align*}
  and again we conclude that \eqref{uniform_one_step_inequality_W}
  holds.

  3) The case where $t\in \{t_i^s,\cdots,t_{i+1}^m-1\}$ for some
  $i\in\N$; that is, when the controller is in the stabilization
  phase. In this case, it follows from \eqref{what_to_obey} that
  \begin{equation*}
    V(\xi(t+1))\leq {\lambda} V(\xi(t))+p_{\max}q.
  \end{equation*}
  Meanwhile, we can conclude from \eqref{def:tau_d_flow} and
  \eqref{def:tau_a_stable} that
  \begin{align*}
    U(\xi(t+1))
    &\leq
      {\Big(\frac{\mu}{\lambda}\Big)^{\tau_d(t)+\frac{1}{\tau}}
      \Big(\frac{\lambda_u}{\lambda}\Big)^{\tau_a(t)+\eta}} 
    \\
    &={\Big(\frac{\mu}{\lambda}\Big)^{\frac{1}{\tau}}
      \Big(\frac{\lambda_u}{\lambda}\Big)^{\eta}}U(\xi(t)) =
      \frac{a}{{\lambda}}U(\xi(t)).  
  \end{align*} 
  Therefore
  \begin{align*}
    W(\xi(t+1))
    &=U(\xi(t+1))V(\xi(t+1))
    \\
    &\leq a U(\xi(t))V(\xi(t))+\frac{a}{
      \lambda}U(\xi(t))p_{\max}q
    \\
    &\leq a W(\xi(t))+\frac{ab}{{\lambda}}p_{\max}q,
  \end{align*}
  and we can conclude the
  inequality~\eqref{uniform_one_step_inequality_W} again.
  
  As a result of the above reasoning, we conclude
  that~\eqref{uniform_one_step_inequality_W} holds for all $t\in\N$.
  % Since $a<1$, we conclude that $W$ is an input-to-state Lyapunov
  % function for the discrete system with respect to $r$. To be
  % precise,
  By the comparison principle~\cite[Lem. 3.4]{HK:02}, we conclude from
  \eqref{uniform_one_step_inequality_W} that
  \[
    W(\xi(t))\leq a^tW(\xi(0)) +\frac{1-a^t}{1-a}\tilde bq\leq
    a^tW(\xi(0))+\frac{\tilde b}{1-a}q. 
  \]
  Recall the definition of the extended state $\xi$ and the fact
  that $p_{\min}\norm{x}\leq V(\xi)\leq p_{\max}\norm{x}$ for any
  $\xi\in\calX$. We have
  \begin{align*}
    p_{\min} \norm{x(t)}\leq
    V(\xi)
    &=\frac{W(\xi(t))}{U(\xi(t))}\leq W(\xi(t))
    \\
    &\leq a^tW(\xi(0))+\frac{\tilde b}{1-a} q
    \\
    &=a^tU(\xi(0))V(\xi(0))+\frac{\tilde b}{1-a} q
    \\
    &\leq a^tbp_{\max}\norm{x_0}+\frac{\tilde b}{1-a} q.
  \end{align*}
  Dividing both sides by $p_{\min}$ yields~\eqref{ISS-estimate},
  proving the result.
  % so \eqref{ISS-estimate} is shown with
  % $c_1:=\big(\frac{\lambda_u}{\lambda}\big)^{\frac{T_0}{2}}\big(\frac{\mu}{\lambda}\big)^{\frac{N_0}{2}}\sqrt{\frac{\overline\lambda}{\underline\lambda}}$,
  % $c_2:=-\frac{\ln a}{2}$ and
  % $\gamma(r):=\sqrt{\frac{br}{\underline\lambda(1-a)}}$.
\end{proof}

\end{document}